\documentclass{article}

\newcommand{\papertitle}{Holonomic Poisson geometry of Hilbert schemes}

\usepackage{titling}

\usepackage{amsfonts,amsmath,amssymb,amsthm}
\usepackage{mathscinet}
\usepackage{xargs}
\usepackage{mathrsfs}

\usepackage[font=small,labelfont=bf]{caption}

\usepackage{pict2e}
\usepackage{geometry}

\usepackage{xypic}
\usepackage{ytableau}

\usepackage{xcolor,hyperref}
\definecolor{darkred}{rgb}{.8,0,0}
\definecolor{tocolor}{rgb}{.1,.1,.1}
\definecolor{urlcolor}{rgb}{.2,.2,.6}
\definecolor{linkcolor}{rgb}{.1,.1,.5}
\definecolor{citecolor}{rgb}{.4,.2,.1}
\definecolor{gray}{rgb}{.8,.8,.8}

\hypersetup{
	backref=true,
	colorlinks=true,
	urlcolor=urlcolor,
	linkcolor=linkcolor,
	citecolor=citecolor,
	pdfauthor = {Mykola Matviichuk and Brent Pym and Travis Schedler},
	pdftitle = {\papertitle}
	}
	
\newcommand{\email}[1]{\href{mailto:#1}{#1}}

\usepackage{todonotes}

\usepackage{aliascnt}

\newcommand{\thdef}[2]{
	\newaliascnt{#1}{theorem}  
	\newtheorem{#1}[#1]{#2}
	\aliascntresetthe{#1}  
	\newtheorem*{#1*}{#2}
	\expandafter\newcommand\expandafter{\csname #1autorefname\endcsname}{#2}
}

\newtheorem{theorem}{Theorem}[section]
\newtheorem*{theorem*}{Theorem}
\thdef{conjecture}{Conjecture}
\thdef{lemma}{Lemma}
\thdef{corollary}{Corollary}
\thdef{proposition}{Proposition}
\theoremstyle{definition}
\thdef{definition}{Definition}

\theoremstyle{remark}
\thdef{remarkx}{Remark}
\thdef{examplex}{Example}

\newenvironment{example}
  {\pushQED{\qed}\examplex}
  {\popQED\endexamplex}

\newenvironment{remark}
  {\pushQED{\qed}\remarkx}
  {\popQED\endremarkx}
  
\newcounter{runexcount}

\newcommand{\defn}[1]{\textbf{\emph{#1}}}

\usepackage{tikz}
\usetikzlibrary{decorations.markings,angles}

\newcommand{\abrac}[1]{\left\langle#1\right\rangle}

\newcommand{\rbrac}[1]{\left(#1\right)}

\newcommand{\set}[2]{\left\{#1 \,\middle|\, #2 \right\}}

\newcommand{\qiso}{\xrightarrow{\sim}}

\newcommand{\CC}{\mathbb{C}}
\newcommand{\CCx}{\mathbb{C}^\times}
\newcommand{\PP}{\mathbb{P}}
\newcommand{\RR}{\mathbb{R}}
\newcommand{\ZZ}{\mathbb{Z}}

\newcommand{\fg}{\mathfrak{g}}
\newcommand{\gln}[2][]{\mathfrak{gl}_{#1}(#2)}
\newcommand{\sln}[1]{\mathfrak{sl}_{#1}}

\newcommand{\aff}[1]{\mathfrak{aff}_{#1}(\CC)}
\newcommand{\Aff}[1]{\mathsf{Aff}_{#1}(\CC)}

\newcommand{\Spec}[1]{\mathsf{Spec}(#1)}

\newcommand{\stab}[2]{\fg_{#1,#2}}
\newcommand{\Stab}[2]{\mathsf{G}_{#1,#2}}
\newcommand{\tormod}[2]{\tau_{#1,#2}}
\newcommand{\sympow}[2][n]{{#2}^{(#1)}}
\newcommand{\Hilb}[2][n]{{#2}^{[#1]}}
\newcommand{\GHilb}[3][n]{(#2,#3)^{[#1]}}
\newcommand{\Utri}[1][]{\U_{\!\SmallTriangularYoungDiagram}^{#1}}

\newcommand{\X}{\mathsf{X}}
\newcommand{\Xu}{\widetilde\X}

\newcommand{\E}{\mathsf{E}} 

\newcommand{\GL}[2][]{\mathsf{GL}_{#1}(#2)}
\newcommand{\Aut}[2][]{\mathsf{Aut}_{#1}\rbrac{#2}}
\newcommand{\G}{\mathsf{G}}

\newcommand{\U}{\mathsf{U}}
\newcommand{\V}{\mathsf{V}}

\newcommand{\D}{\mathsf{D}}

\newcommand{\Du}{\widetilde\D}
\newcommand{\Y}{\mathsf{Y}}
\newcommand{\Z}{\mathsf{Z}}

\newcommand{\bL}{\mathsf{L}}
\newcommand{\W}{\mathsf{W}}
\newcommand{\bS}{\mathsf{S}}
\newcommand{\Tot}[1]{\mathsf{Tot}(#1)}
\newcommand{\tb}[2][]{\mathsf{T}_{#1}#2}
\newcommand{\tc}[2][]{\mathbb{T}_{#1}#2} 
\newcommand{\ctc}[2][]{\mathbb{L}_{#1}#2} 
\newcommand{\ctb}[2][]{\mathsf{T}^*_{#1}#2}

\newcommand{\ps}{\sigma}
\newcommand{\hilbps}[1][n]{\ps^{[#1]}}
\newcommand{\hilbpstor}[1][n]{\ps^{[#1]}_{\Delta} }
\newcommand{\symps}[1][n]{\sympow{\ps}}

\newcommand{\tf}{\mathrm{tf}}
\newcommand{\vir}{{\mathrm{vir}}}
\newcommand{\reg}{{\mathrm{reg}}}
\newcommand{\sing}{{\mathrm{sing}}}
\DeclareMathOperator{\codim}{codim}
\DeclareMathOperator{\coker}{coker}
\DeclareMathOperator{\fibre}{fibre}
\DeclareMathOperator{\img}{image}
\DeclareMathOperator{\Sym}{Sym}
\newcommand{\id}{\mathrm{id}}

\newcommand{\derotimes}[1][]{%
  \mathbin{\mathop{\otimes}\displaylimits_{#1}^{L}}%
}
\newcommand{\tensor}[1][]{%
  \mathbin{\mathop{\otimes}\displaylimits_{#1}}%
}
\newcommand{\hcap}{\mathbin{\mathop{\cap}\displaylimits^{h}}}

\newcommand{\Perf}[1]{\mathsf{Perf}(#1)}
\newcommand{\Perfo}[1]{\mathsf{Perf}_0(#1)}

\newcommandx{\cTor}[3][1=\bullet,2=]{\mathcal{T}or^{#2}_{#1}(#3)}
\newcommandx{\Tor}[3][1=\bullet,2=]{\mathsf{Tor}^{#2}_{#1}(#3)}
\newcommand{\cHom}[2][]{\mathcal{H}om_{#1}(#2)}
\newcommand{\cREnd}[2][]{\mathcal{RE}nd_{#1}(#2)}

\newcommand{\RSect}[1]{\mathsf{R\Gamma}(#1)}
\newcommandx{\Ext}[3][1=,2=]{\mathsf{Ext}^{#1}_{#2}(#3)}
\newcommand{\Hom}[2][]{\mathsf{Hom}_{{#1}}(#2)}
\newcommand{\End}[2][]{\mathsf{End}_{#1}(#2)}

\newcommand{\cT}[1]{\mathcal{T}_{#1}} 

\newcommand{\der}[2][\bullet]{\mathscr{X}^{#1}_{#2}} 
\newcommand{\forms}[2][\bullet]{\Omega^{#1}_{#2}} 
\newcommand{\logforms}[3][\bullet]{\Omega^{#1}_{#2}(\log #3)} 
\newcommandx{\hilblogforms}[4][1=\bullet,2=n]{\logforms[#1]{\Hilb[#2]{#3}}{\Hilb[#2]{#4}}} 

\newcommand{\can}[1]{\mathcal{K}_#1} 
\newcommand{\acan}[1]{\mathcal{K}^\vee_#1} 

\newcommand{\cO}[1]{\mathcal{O}_{#1}} 
\newcommand{\cI}{\mathcal{I}} 
\newcommand{\m}{\mathfrak{m}} 

\newcommand{\sD}[2][]{\mathcal{D}^{#1}_{#2}} 
\newcommand{\charps}{\mathrm{Char}(\ps)}
\newcommand{\charhilbps}[1][]{\mathrm{Char}(\hilbps[#1])}
\newcommand{\cE}{\mathcal{E}}
\newcommand{\cF}{\mathcal{F}}
\newcommand{\koszul}[1]{\mathcal{B}_{#1}}

\newcommand{\coH}[2][\bullet]{\mathsf{H}^{#1}(#2)}
\newcommand{\Hlgy}[2][\bullet]{\mathsf{H}_{#1}(#2)}
\newcommand{\scoH}[2][\bullet]{\mathcal{H}^{#1}(#2)}
\newcommand{\Hps}[2][\bullet]{\mathsf{H}^{#1}_\ps(#2)} 
\newcommandx{\Hhilbps}[3][1=\bullet,2=n]{\mathsf{H}^{#1}_{\hilbps[#2]}(\Hilb{#3})} 

\newcommand{\dd}{\mathrm{d}}
\newcommand{\dlog}[1]{\frac{\mathrm{d}#1}{#1}}
\newcommand{\dps}{\dd_{\ps}}
\newcommand{\dhilbps}[1][n]{\dd_{\hilbps[#1]}}
\newcommand{\cvf}[1]{\partial_{#1}}

\newcommand{\fibprod}{\mathop{\times}}

\newcommand{\syz}{S}
\newcommand{\minsyz}{M(\syz)}

\newcommand{\TriangularYoungDiagram}{\begin{picture}(9,9) 
\put(0,8){\line(0,-1){8}} \put(0,0){\line(1,0){8}} 
\put(8,0){\line(-1,1){8}} \end{picture}}
\newcommand{\SmallTriangularYoungDiagram}{{\,\,\begin{picture}(6,6) 
\put(0,5){\line(0,-1){5}} \put(0,0){\line(1,0){5}} 
\put(5,0){\line(-1,1){5}} \end{picture}}}
\newcommand{\TYD}{\TriangularYoungDiagram}

\newcommand{\hc}[1]{\mathrm{hc}(#1)}
\newcommand{\tlambda}{\widetilde{\lambda}}
\newcommand{\tmu}{\widetilde{\mu}}
\newcommand{\orbdat}[3][]{\mathsf{OrbDat}_{#1}(#2,#3)}

\newcommand{\te}{\widetilde{e}}
\newcommand{\tJ}{\widetilde{J}}

\newcommand{\ES}{Ellingsrud--Str\o{}mme\ }

\begin{document}

\title{\papertitle}
\author{Mykola Matviichuk\thanks{Imperial College London, \email{m.matviichuk@imperial.ac.uk}} \and Brent Pym\thanks{McGill University, \email{brent.pym@mcgill.ca}} \and Travis Schedler\thanks{Imperial College London, \email{t.schedler@imperial.ac.uk}}}
\maketitle
\vspace{-0.5cm}
\abstract{
  We undertake a detailed study of the  geometry of Bottacin's Poisson structures on Hilbert schemes of points in Poisson surfaces, i.e.~smooth complex surfaces equipped with an effective anticanonical divisor.  We focus on three themes that, while logically independent, are linked by the interplay between (characteristic) symplectic leaves and deformation theory.  Firstly, we construct the symplectic groupoids of the Hilbert schemes and develop the classification of their symplectic leaves, using the methods of derived symplectic geometry.  Secondly, we establish local normal forms for the Poisson brackets, and combine them with a toric degeneration argument to verify that Hilbert schemes satisfy our recent conjecture characterizing holonomic Poisson manifolds in terms of the geometry of the modular vector field.  Finally, using constructible sheaf methods, we compute the space of first-order Poisson deformations when the anti-canonical divisor is reduced and has only quasi-homogeneous singularities.  (The latter is automatic if the surface is projective.)  Along the way, we find a tight connection between the Poisson geometry of the Hilbert schemes and the finite-dimensional Lie algebras of affine transformations, which is mediated by syzygies.  In particular, we find that the Hilbert scheme has a natural subvariety that serves as a global counterpart of the nilpotent cone, and we prove that the Lie algebras of affine transformations have holonomic dual spaces---the first such series of Lie algebras to be discovered.
   }
   
{
\setcounter{tocdepth}{2}
\renewcommand{\baselinestretch}{0.9}
\footnotesize
\tableofcontents 
}

\section{Introduction}
Let $\X$ be a Poisson surface, i.e.~a smooth connected complex variety or analytic space of dimension two, equipped with a holomorphic Poisson bivector (a section of the anticanonical line bundle). In \cite{Bottacin1998}, Bottacin constructed a natural Poisson structure on the Hilbert scheme $\Hilb{\X}$ parameterizing zero-dimensional subschemes of $\X$, generalizing the celebrated symplectic structure on the Hilbert schemes of K3 and abelian surfaces due to Beauville~\cite[\S6]{Beauville1983} and Mukai~\cite{Mukai1984}.  Concretely, this Poisson manifold is obtained by resolving the singularities of the symmetric power $\sympow{\X} = \X^n/\bS_n$, but it also has a natural interpretation as a moduli space.  Our goal in this paper is to develop some of the beautiful geometry of Bottacin's Poisson structures, focusing on the following three themes:
\begin{enumerate}
\item symplectic leaves and symplectic groupoids
\item local normal forms and holonomicity
\item deformation theory
\end{enumerate}
In so doing, we encounter connections with a range of topics, from combinatorial linear algebra and Lie theory, to D-modules and derived algebraic geometry.

As we shall see, the three themes listed above are closely intertwined on a conceptual level.  However the results we establish for each of them (summarized in \autoref{sec:intro-sympl} through \autoref{sec:intro-defthy} below, respectively) are based on different techniques and are essentially logically independent.  Thus, for instance, while we make use of the formalism of derived geometry and shifted symplectic structures for the first theme, and this informs our intuition throughout the paper, the results in the second theme are proven using classical techniques of Poisson geometry and toric degeneration, and those in the third theme use constructible sheaves.

\subsection{Symplectic leaves and the Hilbert groupoid}
\label{sec:intro-sympl}

Bottacin's description of $\Hilb{\X}$ as a moduli space gives rise to an explicit description of the Poisson bivector as a pairing on Ext groups, which when combined with a deformation-theoretic argument, leads to a modular description of its symplectic leaves.  Namely, the vanishing locus of the Poisson structure on $\X$ is a curve $\D \subset\X$, and two elements $\Z_1,\Z_2 \in \Hilb{\X}$ lie on the same symplectic leaf if and only if the restrictions of their ideal sheaves to $\D$ are isomorphic as $\cO{\D}$-modules; see, e.g.~\cite{Pym2022,Rains2016,Rains2019}, where such results are established in the more general context of moduli spaces of perfect complexes on noncommutative surfaces, although for Hilbert schemes we obtain this by different means below, as a direct consequence of \autoref{thm:groupoid}.

It is natural to ask for a geometric (rather than module-theoretic) interpretation of the leaves.  For instance, one can check that the map sending an element  $\Z \in \Hilb[n]{\X}$ to the intersection $\Z \cap \D$ is invariant under Hamiltonian flow, and hence if two elements $\Z_1,\Z_2$ lie on the same symplectic leaf, we must have $\Z_1 \cap \D = \Z_2 \cap \D$.  Over the open set of $\Hilb[n]{\X}$ corresponding to collections of $n$ distinct points of $\X$, this characterizes the symplectic leaves completely, but the situation in general is not so simple.  Indeed, even for a smooth curve $\D$, there exist pairs $\Z_1,\Z_2 \in \Hilb[n]{\X}$ that have the same scheme-theoretic intersection with $\D$, but lie on different symplectic leaves.  

The reason for this subtlety is that the intersection $\Z \cap \D$, if non-empty, is never transverse.  Hence, it is natural to consider the derived intersection $\Z \hcap \D$, which enhances the naive algebra of functions $\cO{\Z\cap \D} \cong\cO{\Z}\tensor[\cO{\X}] \cO{\D}$ to the dg algebra $\cO{\Z\hcap \D} :=\cO{\Z}\derotimes[\cO{\X}] \cO{\D}$, whose cohomology gives the Tor modules $\cTor[\bullet][\cO{\X}]{\cO{\Z},\cO{\D}}$.  An important feature is that the derived intersection $\Z\hcap \D$ can have nontrivial automorphisms that preserve the embedding $\Z \hcap \D \hookrightarrow \D$, so that the notion of the equality $\Z_1 \cap \D = \Z_2 \cap \D$ of ordinary scheme-theoretic intersections is enhanced to a whole space of equivalences $\Z_1 \hcap \D \sim \Z_2 \hcap \D$ of derived subschemes of $\D$, which we call \defn{$\D$-equivalences} of the pair $(\Z_1,\Z_2)$; these turn out to be the same as determinant-one $\cO{\D}$-module isomorphisms between the restrictions of their ideal sheaves to $\D$.   These equivalences can be composed, giving a groupoid $\GHilb[n]{\X}{\D}$ whose objects are elements of $\Hilb[n]{\X}$ and whose morphisms are homotopy classes of $\D$-equivalences.  We call this groupoid the \defn{Hilbert groupoid of the pair $(\X,\D)$}, and develop its structure in Sections \ref{sec:groupoids} and \ref{sec:pois} below.  We summarize these results as follows:

\begin{theorem}\label{thm:groupoid}
Let $\X$ be a smooth complex surface, and let $\D\subset \X$ be any curve (which may be singular or even non-reduced).  Then $\GHilb[n]{\X}{\D}$  naturally has the structure of a smooth groupoid (i.e.~a complex Lie groupoid), which is birationally equivalent to the product $\Hilb[n]{\X}\times\Hilb[n]{\X}$.  If, in addition, the curve $\D$ is anticanonical, then a choice of Poisson structure on $\X$ whose vanishing locus is $\D$ endows $\GHilb[n]{\X}{\D}$ with a natural symplectic structure, making it into a symplectic groupoid that integrates Bottacin's Poisson structure on $\Hilb[n]{\X}$.  In particular, the symplectic leaves are the connected components of the $\D$-equivalence classes, and these are locally closed subvarieties of $\Hilb{\X}$.
\end{theorem}

We remark that although the output of the construction is an ordinary smooth variety, the proof we present takes full advantage of the derived geometry toolbox, particularly foundational results on moduli stacks of sheaves and shifted symplectic structures from~\cite{Brav2021,Calaque2015a,Pantev2013,Toen2007,Schurg2015}.  Such an approach naturally (and rather easily) constructs $\GHilb{\X}{\D}$ as a derived stack; our main new addition is a representability statement, which shows that this derived stack is in fact a smooth variety.  Combining this with normal forms for sheaves on smooth and nodal curves, we obtain a classification of symplectic leaves in the case $\D$ has only nodal singularities.  For smooth curves,
this problem  was also treated in \cite[Section 11.2]{Rains2016}; we expand on the results in \emph{op.~cit.}, emphasizing connections with the combinatorics of Young diagrams and monomial ideals.

\subsection{Local models and holonomicity}
\label{sec:intro-local}
In \autoref{sec:local} we turn to our second theme, namely the construction of explicit local models for the Poisson manifold $\Hilb{\X}$ and its symplectic groupoid.

The guiding principle here is that the derived intersection $\Z\hcap \D$ can be described in purely classical terms using the  Hilbert--Burch resolution of the structure sheaf $\cO{\Z}$, which expresses the ideal defining $\Z$ locally as the maximal minors of a $(k+1)\times k$ matrix with entries in $\cO{\X}$ (the ``syzygy matrix''). As explained by Ellingsrud--Str\o{}mme~\cite{Ellingsrud1988} one can construct natural coordinates on $\Hilb[n]{\X}$ by choosing suitable normal forms for the syzygy matrices.  We use these coordinates to establish a tight connection between the local structure of the Poisson manifolds $\Hilb[n]{\X}$ and the Lie algebras $\aff{k} \cong \gln[k]{\CC}\ltimes \CC^k$ of affine transformations of $\CC^k$, for $k \ge 1$.  Indeed, by passing to open subsets and taking products, one can reduce the problem of understanding the local structure of the Hilbert schemes to the case of open balls around the elements $\Z = p(k) \in \Hilb[n]{\X}$ that correspond to the $k$-th order neighbourhoods of points $p \in \X$, where $k \ge 1$, for which we establish the following description.
\begin{theorem}\label{thm:local}
A choice of log Darboux coordinates on $\X$ at $p$ induces an analytic Poisson isomorphism from a neighbourhood of $p(k)$ in $\Hilb[k(k+1)/2]{\X}$ to a neighbourhood of the origin in $\aff{k}^\vee$, where the latter is equipped with the following Poisson structure:
\begin{itemize}
\item a constant symplectic structure with explicit Darboux coordinates, if $p \in \X\setminus \D$;
\item the Lie--Poisson structure, if $p$ is a smooth point of $\D$; or
\item a Lie-compatible quadratic Poisson structure, if $p$ is a node of $\D$.
\end{itemize}
\end{theorem}
As a consequence,  when $\D$ is smooth, the symplectic groupoid $\aff{k}^\vee \rtimes \Aff{k}$ of $\aff{k}^\vee$ locally integrates the Poisson structure; in fact, we show that this a local model for the groupoid $\GHilb[n]{\X}{\D}$ constructed in \autoref{thm:groupoid}, when $n=\tfrac{1}{2}k(k+1)$.  
Another interesting aspect of this case is that
 we can identify a natural subvariety in $\Hilb[k(k+1)/2]{\X}$ that globalizes the embedding  $\gln[k]{\CC}^\vee \subseteq \aff{k}^\vee$, namely the locus $\W$ of schemes whose intersection with $\D$ has length $k$.  The resulting map $\W \to \Sym^k \D$ locally encodes the universal Poisson deformation of the nilpotent cone, giving a link to Springer theory---in particular, the fibre over $k \cdot p$ is a global version of the nilpotent cone (see \autoref{r:groth-spr}). 

Of the three cases in \autoref{thm:local}, the third one is the most subtle. The quadratic Poisson structure in question seems to be new, and has the remarkable property that it admits a canonical isotrivial toric degeneration.  The construction of this degeneration involves a careful analysis  of the combinatorics of the weights of a natural torus action on $\aff{k}$, which we recast visually as a sort of ``game of dominoes''; see \autoref{sec:toric-degen}.

 As an application of our local models, we show that the Hilbert schemes give new examples of Poisson manifolds satisfying a natural and subtle nondegeneracy condition, called holonomicity.  (This was the problem that motivated the broader investigation we present in this paper.)  The notion of holonomicity was introduced by the second- and third-named authors in~\cite{Pym2018}, motivated by deformation theory; it is roughly equivalent to the statement that for each point in the Poisson manifold, the space of derived Poisson deformations of its germ is finite-dimensional.  This notion was further developed and applied in our joint work \cite{Matviichuk2020}, where we emphasized the importance of a certain collection of symplectic leaves, which we call ``characteristic'': these are the symplectic leaves that are preserved by the flow of the modular vector field.  In particular, while the original definition of holonomicity involves $\sD{}$-modules, we conjectured that it is equivalent to a simple geometric criterion:
\begin{conjecture}[\cite{Matviichuk2020}]\label{conj:holonomicity}
A Poisson manifold is holonomic if and only if its germ at any point has only finitely many characteristic symplectic leaves.
\end{conjecture} 

 In the present work, we show that the symplectic leaf through a point $\Z \in \Hilb[n]{\X}$ is characteristic if and only if the corresponding derived intersection $\Z\hcap \D$ is preserved by the flow of the modular vector field on $\D$. Once again, this can be turned into a concrete local computation in examples without relying on derived techniques.  Combining such concrete calculations with the results in \cite{Matviichuk2020,Pym2018}, the aforementioned toric degenerations, and some subtle combinatorial linear algebra, we establish our conjecture in (almost) complete generality for  Hilbert schemes: 
\begin{theorem}\label{thm:holonomic}
For a Poisson surface $\X$ whose vanishing locus is the anticanonical divisor $\D\subset \X$, the following statements are equivalent:
\begin{enumerate}
\item For every $n\ge0$, the induced Poisson structure on the Hilbert scheme $\Hilb[n]{\X}$ is holonomic.
\item  For every $n\ge0$, the germ of  $\Hilb[n]{\X}$ at any point has only finitely many characteristic symplectic leaves.
\item The only singularities of $\D$ are nodes.
\end{enumerate}
\end{theorem} 
Note that the order of the quantifiers in statements 1 and 2 is important: for $n< 6$ there exist some exceptions for which $\Hilb[n]{\X}$ has finitely many characteristic leaves, but $\D$ has non-nodal singularities.  However, we can treat most of these cases directly, leaving only one case left to determine whether \autoref{conj:holonomicity} holds for all Hilbert schemes, namely the case when $n = 5$ and $\D$ has only  $A_2$-singularities (cusps).

On the other hand, combining \autoref{thm:holonomic} with \autoref{thm:local}, we obtain the first (and so far only known)  infinite series of holonomic Lie algebras:

\begin{corollary}
The Lie--Poisson structure on $\aff{k}^\vee$ is holonomic for any $k\ge1$, and in particular its Poisson cohomology is finite-dimensional.
\end{corollary}
It would be interesting to know if there are other series of Lie algebras with holonomic duals.  Note such Lie algebras must have an open coadjoint orbit, i.e.~they are Frobenius.  However, many of the Frobenius Lie algebras known to us are not holonomic.

\subsection{Deformation theory}
\label{sec:intro-defthy}
We close the paper in \autoref{sec:def-thy} by discussing the  Poisson deformations of $\Hilb[n]{\X}$.  Such deformations are governed by the Poisson cohomology $\Hhilbps{\X}$, which in turn is controlled by the geometry of the open symplectic leaf and the codimension-two characteristic leaves.  We compute the Poisson cohomology in low degrees under the additional assumption that  $\D$ has only quasi-homogeneous singularities (which includes all cases where
$\X$ is projective and $\D$ is reduced). This extends earlier work of Ran~\cite{Ran2016}, who treated the case where $\D$ is smooth. In particular, we determine space of first-order deformations (the second Poisson cohomology):
\begin{theorem}
  Assume that $\X$ is connected, and that the anticanonical divisor $\D$ is reduced and has only quasi-homogeneous singularities.  Then the second Poisson cohomology of $\Hilb{\X}$, is given, for all $n \ge 2$, by
\begin{align*}
\Hhilbps[2]{\X} &\cong  \coH[2]{\Hilb{(\X\setminus \D)};\CC} \oplus \coH[0]{\acan{\X}|_{\D_\sing}} \\
&\cong \Hps[2]{\X} \oplus \wedge^2 \coH[1]{\X\setminus \D;\CC} \oplus  \CC\cdot [\E]
\end{align*}
where $\acan{\X}$ is the anticanonical bundle and $[\E]$ is the Poisson Chern class of the exceptional divisor of the Hilbert--Chow morphism.
\end{theorem}

This decomposition of the deformation space has the following interpretation.  The summand $\Hps[2]{\X}$ corresponds to deformations of $\Hilb{\X}$ induced by applying Bottacin's construction to a deformation of $\X$ itself.  Meanwhile the summand $\wedge^2\coH[1]{\X\setminus \D;\CC}$ gives deformations in which the symplectic form on the open symplectic leaf $\Hilb{(\X\setminus \D)}$  is deformed by the pullback of a closed two-form along the Albanese map $\Hilb{(\X\setminus \D)} \to \coH[1]{\X\setminus \D;\CC}^\vee / \Hlgy[1]{\X\setminus \D;\ZZ}$.  Finally, the deformations in the direction $\CC\cdot[\E]$ should correspond to Hilbert schemes of \emph{noncommutative} surfaces given by deformation quantization of $(\X,\ps)$ as in \cite{Rains2016,Rains2019}.  In the case $\X=\PP^2$ and $\D$, this interpretation was established by Hitchin~\cite{Hitchin2012}, who proved that the resulting deformations are exactly the Hilbert schemes of Sklyanin algebras introduced by Nevins--Stafford~\cite{Nevins2007}.

\subsubsection*{Acknowledgements}

We thank Marco Gualtieri, Jacques Hurtubise, Anton Izosimov, Eric Rains and Pavel Safronov for helpful conversations. BP is supported by a Discovery Grant from the Natural Sciences and Engineering Research Council of Canada; a Research Support for New Academics grant from the Fonds de recherche du Qu\'ebec (Nature et technologies); and a Faculty Startup Grant from McGill University. MM is supported by Chapman Fellowship. MM thanks Jacques Hurtubise, Niky Kamran and CIRGET for partially funding his postdoctoral position at McGill, during which most of his contribution to this paper took place. TS thanks the Max Planck Institute for Mathematics in Bonn and the Hausdorff Institute of Mathematics in Bonn for excellent working conditions during preliminary stages of work on this project.

\section{Hilbert groupoids of curves in surfaces}
\label{sec:groupoids}

\subsection{Hilbert schemes of surfaces}
\label{sec:Hilb}
Throughout this paper, by a \defn{$\CC$-scheme} we mean a scheme of finite type over $\CC$ or a complex analytic space.  By a \defn{smooth (complex) surface}, we mean a smooth $\CC$-scheme of dimension two.

Let $\X$ be a smooth complex surface and suppose that $n \in \ZZ_{>0}$ is a positive integer.  We denote by $\Hilb[n]{\X}$ the Hilbert scheme (or Douady space) parameterizing zero-dimensional subschemes of $\X$ of total length $n$.  There are many equivalent presentations of this space, each of which will be useful in what follows:
\begin{itemize}
\item As a functor of points: if $\bS$ is a $\CC$-scheme, the set of maps from $\bS$ to $\Hilb[n]{\X}$ is given by
\[
\Hilb[n]{\X}(\bS) := \{\textrm{subschemes of }\bS \times \X\textrm{ that are finite and flat of length }n\textrm{ over }\bS\}
\]
It is a theorem of Fogarty~\cite{Fogarty1968} that  that this functor is representable by a smooth scheme.
\item As a resolution of singularities: let $\sympow{\X}$ be the $n$-th symmetric power of $\X$, i.e.~the quotient of $\X^n$ by the permutations of the factors; it is a singular variety when $n \ge 2$.  There is a natural map $\Hilb[n]{\X}\to \sympow{\X}$, called the Hilbert--Chow morphism, which sends a length-$n$ subscheme $\Z\subset\X$ to its support (the collection of points in $\Z$, counted with multiplicity).  This map is a resolution of singularities that is crepant~\cite[proof of Proposition 5]{Beauville1983} and strictly semi-small~\cite[Theorem 2.12]{Iarrobino1977}.
\item As a moduli space of sheaves: a point $\Z \in \Hilb[n]{\X}$ is equivalent to the data of its ideal sheaf $\cI_\Z \subset \cO{\X}$.  We may view $\cI_\Z$ as a perfect complex of coherent sheaves on $\X$.  This complex is canonically isomorphic to $\cO{\X}$ away from $\Z$, and this induces a canonical trivialization of the determinant $\det{\cI_\Z}$ away from $\Z$, which extends to all of $\X$  since $\X$ is smooth and $\Z$ has codimension two.  In this way, we identify $\Hilb[n]{\X}$ with an open substack of the derived stack $\Perfo{\X}$ of perfect complexes with trivialized determinant~\cite{Schurg2015}.  Note that the tangent complex of $\Hilb[n]{\X}$ at $\Z$ is then given by
\[
\tc[\Z]\Hilb[n]{\X} \cong \tc[\cI_\Z]\Perfo{\X} \cong \RSect{\cREnd{\cI_\Z}_0}[1] 
\]
where $\cREnd{\cI_\Z}_0$ is the homotopy fibre of the Illusie trace map $\cREnd{\cI_\Z} \to \cO{\X}$.  The cohomology of this complex is concentrated in degree zero, giving the tangent space of $\Hilb[n]{\X}$ at $\Z$ as
\begin{align}
\tb[\Z]{\Hilb{\X}} \cong \coH[1]{\RSect{\cREnd{\cI_\Z}_0}} =: \Ext[1][\X]{\cI_\Z,\cI_\Z}_0 \cong \Hom[\X]{\cI_\Z,\cO{\Z}} \label{eq:Hilb-tangent}
\end{align}
\end{itemize}

\subsection{Derived intersection of subschemes with curves}

If $\X$ is a smooth surface, a \defn{curve in $\X$} is an effective divisor $\D\subset \X$; unless otherwise specified, we allow the possibility that $\D$ is singular, or even non-reduced.

Suppose that $ i : \D \hookrightarrow \X$ is a curve.  If $\Z \subset \X$ is a length-$n$ subscheme, then the scheme-theoretic intersection $\Z\cap \D \subset \D$ is a closed subscheme of $\D$ of length at most $n$, with structure sheaf $\cO{\Z \cap \D} \cong \cO{\Z}\tensor_{\cO{\X}} \cO{\D}$.
The \defn{derived intersection} $\Z\hcap \D$, is the derived subscheme of $\D$ obtained as the spectrum of the sheaf of commutative dg $\cO{\D}$-algebras $\cO{\Z\hcap \D} := \cO{\Z}\derotimes[\cO{\X}]\cO{\D}$, where $\derotimes$ denotes the derived tensor product.  In other words, we consider the dg algebra $Li^*\cO{\Z}$, where
\[
Li^* : \Perf{\X} \to \Perf{\D} \qquad Li^* \cE := i^{-1}\cE \derotimes[i^{-1}\cO{\X}]\cO{\D}
\]
is the derived restriction of perfect complexes. Resolving $\cO{\D}$ by the Koszul complex
\[
\cO{\D} \cong \rbrac{ \cO{\X}(-\D) \to \cO{\X} }
\]
we have concretely that
\begin{align}
\cO{\Z\hcap \D} \cong (\cO{\X}(-\D)|_\Z \to \cO{\Z}). \label{eq:koszul}
\end{align}
with $\cO{\Z}$ placed in cohomological degree zero. 

\begin{remark}
The number of points of intersection, counted with multiplicity, is equal to
\begin{align}
\# (\Z \cap \D) := \dim_\CC \coH[0]{\cO{\Z\cap \D}}  \label{eq:num-points-def}
\end{align}
by definition. 
For the derived intersection,
it is natural to consider
the virtual count of points, defined by the Euler characteristic of $\cO{\Z \hcap \D}$, rather than its zeroth cohomology.  Using the Koszul resolution \eqref{eq:koszul}, we find
\[
\#_{\vir} (\Z \hcap \D) := \chi(\coH{\cO{\Z\hcap D}}) = \dim_\CC \coH[0]{\cO{\Z}} - \dim_\CC\coH[0]{\cO{\X}(-\D)|_\Z} = 0
\]
for all $\Z$, since $\Z$ is zero-dimensional and $\cO{\X}(-\D)|_\Z$ is an invertible $\cO{\Z}$-module, in accordance with the general principle that derived intersection numbers are deformation-invariant.
\end{remark}

We now determine when two elements $\Z_1,\Z_2 \in \Hilb[n]{\X}$ have equivalent derived intersection with $\D$.  Note that since the ideal sheaf $\cI_\Z$ is torsion-free, we have $\cTor[j]{\cI_\Z,\cO{\D}} = 0$ for $j > 0$, so the natural map $Li^*\cI_\Z \to i^*\cI_\Z$ comparing the derived and ordinary pullbacks is a quasi-isomorphism.  Hence we obtain an exact triangle
\begin{align}
\xymatrix{
i^*\cI_\Z \ar[r] & \cO{\D} \ar[r] & Li^*\cO{\Z}.
} \label{eq:derint-triangle}
\end{align}
of perfect complexes on $\D$, corresponding to the exact sequence
\[
\xymatrix{
0\ar[r] & \cTor[1][\cO{\X}]{\cO{\Z},\cO{\D}} \ar[r] & i^*\cI_\Z \ar[r] & \cO{\D} \ar[r] & \cO{\Z\cap \D} \ar[r] & 0
}
\]
of coherent sheaves.  Note that $\cTor[1]{\cO{\Z},\cO{\D}}$ is a  torsion module on $D$. Since $\cO{\D}$ is torsion-free (by definition), the image
\[
i^*\cI_\Z^\tf := \img(i^*\cI_\Z \to \cO{\D}) \cong i^*\cI_\Z/\cTor[1][\cO{\X}]{\cO{\Z},\cO{\D}}
\]
is the maximal torsion-free quotient of $i^*\cI_\Z$.  We then have the following.

\begin{lemma}\label{lem:D-equiv-equivs}
For any $(\Z_1,\Z_2)\in\Hilb[n]{\X}$, the following spaces are canonically homotopy equivalent:
\begin{enumerate}
\item The space of equivalences $\Z_1\hcap \D \qiso \Z_2\hcap \D$ of derived subschemes of $\D$.
\item The space of quasi-isomorphisms $Li^*\cO{\Z_1} \qiso Li^*\cO{\Z_2}$ of commutative differential graded $\cO{\D}$-algebras
\item The space of quasi-isomorphisms $Li^*\cO{\Z_1} \qiso Li^* \cO{\Z_2}$ of perfect complexes of $\cO{\D}$-modules, commuting up to coherent homotopy with the natural maps $Li^*\cO{\Z_1} \leftarrow \cO{\D} \rightarrow Li^*\cO{\Z_2}$
\item\label{item:D-equiv-det} The space of equivalences $i^*\cI_{\Z_1} \qiso i^*\cI_{\Z_2}$ in $\Perfo{\D}$.
\item\label{item:D-equiv-ideal} The discrete set of isomorphisms $i^*\cI_{\Z_1}  \qiso i^*\cI_{\Z_2}$ commuting with the natural maps $i^*\cI_{\Z_1} \rightarrow \cO{\D} \leftarrow  i^*\cI_{\Z_2}$.
\item\label{item:D-equiv-tor} The discrete set of isomorphisms $i^*\cI_{\Z_1}  \qiso i^*\cI_{\Z_2}$ that induce the identity map on the maximal torsion-free quotients $i^*\cI_{\Z_j}^\tf \subset \cO{\D}$.
\end{enumerate}
In particular, all of these spaces are equivalent to their discrete sets of connected components.
\end{lemma}

\begin{proof}
The spaces in 1 and 2 are equivalent by definition of the derived intersection. 

 For the equivalence of 2 and 3, choose an arbitrary cofibrant resolution $\cE^\bullet \to \cO{\Z_1}$ of $\cO{\Z_1}$ as a sheaf of $\cO{\X}$-dg algebras concentrated in degrees $\le 0$, with degree zero term $\cE^0 = \cO{\X}$. The dg $\cO{\D}$-algebra $i^* \cE \cong \cE^\bullet \tensor[\cO{\X}] \cO{\D}$ is then also cofibrant.  Meanwhile, we may resolve $\cO{\D}$ by the Koszul complex $\koszul{\D} := (\cO{\X}(-\D) \to \cO{\X})$.   The space of equivalences $Li^*\cO{\Z_1}\cong Li^*\cO{\Z_2}$ of dg $\cO{\D}$-algebras is then equivalent to the space of dg $\cO{\D}$-algebra maps $i^*\cE^\bullet \to i^*(\cO{\Z_2} \otimes \koszul{\D})$.  By degree considerations, such a dg algebra map is equivalent to a map of complexes whose degree zero part is the natural projection $\cO{\D} \to \cO{\Z_1}$, as desired.  Note further that the full space of equivalences is modelled by the simplicial set of dg algebra maps $i^*\cE \to  i^*(\cO{\Z_2} \otimes \koszul{\D}) \otimes \forms{\Delta}$ where $\forms{\Delta}$ is the simplicial dg algebra of polynomial differential forms on simplices.  Since the latter is concentrated in non-negative degrees, the same degree considerations show that the higher simplices of spaces 2 and 3 are also the same, as desired.

The equivalence of 3 and 4 follows from the exact triangle \eqref{eq:derint-triangle}, using the fact that the trivialization of the determinant of $\cI_{\Z}$  is defined by the inclusion $\cI_{\Z} \hookrightarrow \cO{\X}$. 

 Finally, 5 and 6 are clearly equivalent by definition of $i^*\cI_\Z^\tf$, and they are equivalent to 4 because for an arbitrary morphism $\phi : \cI_{\Z_1} \to \cI_{\Z_2}$ the diagram
\[
\xymatrix{
\cI_{\Z_1} \ar[r]^-{\phi}\ar[d] \ar[r] & \cI_{\Z_2} \ar[d] \\
\cO{\D} \ar[r]^-{\det \phi} \ar[r] & \cO{\D}
}
\]
is commutative.
\end{proof}

\begin{remark}
The proposition holds much more generally, with essentially the same proof.  Namely, statements 1, 2, 3, 5 and 6 are equivalent for any pair of subschemes $\Z_1,\Z_2$ in a scheme $\X$ with an effective Cartier divisor $\D$.  For the equivalence with 4, we need to speak about trivialization of the determinant being determined away from codimension two subsets, and for this we need to assume in addition that $\X$ is $S_2$ (e.g. smooth or more generally normal) and that the subschemes $\Z_1$, $\Z_2 \subset \X$ have codimension at least two.
\end{remark}

\begin{definition}
We refer to any of the equivalent discrete sets in \autoref{lem:D-equiv-equivs} as the \defn{set of $\D$-equivalences from $\Z_1$ to $\Z_2$}. We say that $\Z_1$ and $\Z_2$ are \defn{$\D$-equivalent} if there exists a $\D$-equivalence from $\Z_1$ to $\Z_2$.
\end{definition}

Evidently $\D$-equivalence defines an equivalence relation on $\Hilb{\X}$.  If we are just interested in the equivalence classes, an a priori weaker notion will suffice:
\begin{lemma}\label{lem:D-equiv-module}
Two elements $\Z_1,\Z_2 \in \Hilb{\X}$ are $\D$-equivalent if and only if the $\cO{\D}$-modules $i^*\cI_{\Z_1}$ and $i^*\cI_{\Z_2}$ are isomorphic.
\end{lemma}
\begin{proof}
By part \ref{item:D-equiv-det} of \autoref{lem:D-equiv-equivs}, it suffices to show that $i^*\cI_{\Z_1}$ and $i^*\cI_{\Z_2}$ are isomorphic as $\cO{\D}$-modules if and only if they are are isomorphic by a map with trivial determinant.  But if $\phi : i^*\cI_{\Z_1} \to i^*\cI_{\Z_2}$ is any isomorphism, then its determinant gives an isomorphism 
\[
\cO{\D} \cong \det{\cI_{\Z_1}} \qiso \det{\cI_{\Z_2}} \cong \cO{\D},
\]
or equivalently an invertible element $g := \det\phi \in \coH[0]{\cO{\D}^\times}$.  Since the rank of an ideal sheaf is equal to one, the determinant is a linear functional on endomorphisms and we have 
\[
\det(g^{-1} \phi) = g^{-1}\det(\phi) = \det(\phi)^{-1}\det(\phi) = 1.
\]
Therefore $g^{-1}\phi : \cI_{\Z_1} \to \cI_{\Z_2}$ is an isomorphism with trivial determinant, as desired.
\end{proof}

\subsection{The Hilbert groupoid}
Evidently $\D$-equivalences between points of $\Hilb[n]{\X}$ can be composed, so that they form the arrows of a groupoid whose objects are the points of $\Hilb[n]{\X}$.  We will now show that this defines a smooth groupoid scheme (or in the analytic context, a Lie groupoid).

More precisely, let
\[
\GHilb[n]{\X}{\D} := \Hilb[n]{\X} \fibprod_{\Perfo{\D}} \Hilb[n]{\X} \subset \Perfo{\X} \fibprod_{\Perfo{\D}} \Perfo{\X}
\]
be the derived stack defined as the homotopy fibre product of $\Hilb[n]{\X}$ with itself along the map sending a subscheme $\Z$ to $i^*\cI_\Z$, where $\cI_\Z$ is the ideal sheaf of $\Z$ equipped with the canonical trivialization of $\det \cI_\Z$.  By definition of the homotopy fibre product, points of $\GHilb[n]{\X}{\D}$ consist of triples $(\Z_1,\Z_2,\lambda)$, where $\Z_1,\Z_2$ are points of $\Hilb[n]{\X}$, and $\lambda : \Z_1\hcap \D \qiso \Z_2 \hcap \D$ is a $\D$-equivalence.  Composing the $\D$-equivalences  gives $\GHilb[n]{\X}{\D}$ the structure of a groupoid in the category of derived stacks.  

\begin{definition}
The groupoid $\GHilb[n]{\X}{\D}$ over $\Hilb[n]{\X}$ is the \defn{$n$-th Hilbert groupoid of $(\X,\D)$}.
\end{definition}

In fact this object is a classical smooth variety:
\begin{theorem}\label{prop:representable}
The derived stack $\GHilb[n]{\X}{\D}$ is a smooth classical scheme. The projections $s,t:\GHilb[n]{\X}{\D} \to \Hilb[n]{\X}$ are smooth morphisms, so that $\GHilb[n]{\X}{\D}$ defines a smooth groupoid scheme over $\Hilb[n]{\X}$ whose orbits are the $\D$-equivalences classes.  Moreover, the induced map $(s,t) :\GHilb[n]{\X}{\D} \to \Hilb[n]{\X} \times \Hilb[n]{\X}$  restricts to an isomorphism over the Zariski open set $\Hilb[n]{(\X\setminus \D)} \times \Hilb[n]{(\X\setminus \D)}$, and is therefore birational.
\end{theorem}

\begin{proof}
To prove representability, we will use the interpretation of $\D$-equivalences from part \ref{item:D-equiv-ideal} of \autoref{lem:D-equiv-equivs}: the fibre of $\GHilb[n]{\X}{\D}$ over a point $(\Z_1,\Z_2) \in \Hilb[n]{\X}\times\Hilb[n]{\X}$ is the set of isomorphisms $\cI_{\Z_1} \otimes \cO{\D} \to \cI_{\Z_2} \otimes \cO{\D}$ commuting with the natural maps to $\cO{\D}$.  Recall that there is a universal ideal sheaf $\cI$ on $\Hilb[n]{\X}\times \X$ whose restriction to any slice $\{\Z\} \times \X \cong \X$ is the ideal defining $\Z$.  Let $\Xu := \Hilb[n]{\X}\times \Hilb[n]{\X} \times \X$ and let $\cI_1,\cI_2 \subset \cO{\Xu}$ be the pullbacks of $\cI$ along the two projections to $\Hilb[n]{\X}\times\X$.  Let $\Du := \Hilb[n]{\X}\times\Hilb[n]{\X}\times \D \subset \Xu$.   Finally, let $p : \Xu \to \Hilb[n]{\X}\times\Hilb[n]{\X}$ be the projection and consider the sheaf $\cE := p_*\cHom{\cI_1,\cI_2}$.  Composition with the map $\cI_2 \to \cO{\Xu}\to\cO{\Du}$ gives a natural map $p_*\cHom{\cI_1 \otimes \cO{\Du},\cI_2\otimes\cO{\Du}} \to p_*\cHom{\cI_1,\cO{\Du}}$ of quasi-coherent sheaves on $\Hilb[n]{\X}\times\Hilb[n]{\X}$ and the morphisms $\cI_1 \otimes \cO{\Du} \to \cI_2\otimes\cO{\Du}$ compatible with the maps to $\cO{\Du}$ form a torsor for the quasi-coherent sheaf
\[
\cE := \ker p_*\cHom{\cI_1 \otimes \cO{\Du},\cI_2\otimes\cO{\Du}} \to p_*\cHom{\cI_1,\cO{\Du}}
\]
It is straightforward to check that the sheaf $\cE$ is coherent, so that its total space $\Tot{\cE}$ is a scheme that is relatively affine over $\Hilb[n]{\X}\times \Hilb[n]{\X}$, and hence so is the torsor in question.  The invertible morphisms then form an open subscheme of the torsor, as desired.

To establish the smoothness of $s,t$, it is now enough to check that their relative tangent complexes $\tc{s},\tc{t}$ at any closed point of $\GHilb[n]{\X}{\D}$ are concentrated in degree zero; note that this statement is independent from the details of the explicit construction of the scheme above, and is indeed proved most easily from the abstract definition as a derived stack.  We will prove the smoothness for $s$; the argument for $t$ is identical.

 We have the following commutative diagram of tangent complexes:
\[
\xymatrix{
\tc{s} \ar[r] \ar[d] & \tc \GHilb[n]{\X}{\D} \ar[d]\ar[r] & s^*\tc\Hilb[n]{\X} \ar@{=}[d]\\
t^*\tc\Hilb[n]{\X}\ar[r] \ar[d]& s^*\tc \Hilb[n]{\X}\oplus t^*\tc \Hilb[n]{\X} \ar[r]\ar[d] & s^*\Hilb[n]{\X} \ar[d] \\
(Li^*)\tc{\Perfo{\D}}\ar@{=}[r]& (Li^*)^*\tc\Perfo{\D} \ar[r] & 0
}
\]
where $\tc{s}$ denotes the relative tangent complex of the map $s : \GHilb[n]{\X}{\D} \to \Hilb[n]{\X}$.  Since homotopy fibre products of stacks give homotopy fibre sequences of tangent complexes, all three  rows and the right two columns are fibre sequences, from which we deduce that 
\begin{align}
\tc{s} &\cong \fibre\rbrac{ t^* \tc{\Hilb[n]{\X}} \to (Li^*)^*\tc\Perfo{\D} }  \label{eq:tc-source}
\end{align}
Suppose that $(\Z_1,\Z_2,\lambda) \in \GHilb[n]{\X}{\D}$ is a closed point.  Let $\cI = \cI_{\Z_2}$ be the ideal defining $\Z_2$.  Identifying the tangent complexes of $\Hilb[n]{\X}\subset \Perfo{\X}$ and $\Perfo{\D}$ in \eqref{eq:tc-source} with the  complexes of traceless derived endomorphisms of $\cI$ and $Li^*\cI$ respectively, we have
\begin{align*}
\tc[(\Z_1,\Z_2,\lambda)]s &\cong \fibre\rbrac{ \RSect{\cREnd[\X]{\cI}_0} \to \RSect{\cREnd[\D]{Li^*\cI}}_0}[1] \\
&\cong \RSect{\cREnd[\X]{\cI}_0(-\D)}[1]
\end{align*}
where the second equivalence is induced by tensoring the complex $\cREnd[\X]{\cI}_0$ with the exact sequence for the quotient $\cO{\D}\cong\cO{\X}/\cO{\X}(-\D)$.  As explained in~\cite[Corollary 2.6]{Bottacin1998}, we have
\[
\cREnd[\X]{\cI}_0[1] \cong \cHom[\X]{\cI,\cO{\Z_2}}
\]
so that
\[
\cREnd[\X]{\cI}_0(-\D)[1] \cong \cHom[\X]{\cI,\cO{\Z_2}} \otimes_{\cO{\X}} \cO{\X}(-\D)
\]
is a skyscraper sheaf concentrated on the zero-dimensional scheme $\Z_2$.  Therefore
\[
\tc[(\Z_1,\Z_2,\lambda)]s \cong \RSect{\cHom{\cI,\cO{\Z_2}} \otimes_{\cO{\X}} \cO{\X}(-\D)} \cong \Hom[\X]{\cI,\cO{\Z_2}(-\D)}
\]
is concentrated in degree zero, as desired.
\end{proof}

\begin{example}\label{ex:blowup-construction}
The construction of the Hilbert groupoid is already interesting in the case when $\D$ is smooth and $n=1$, where it recovers a known construction of Lie groupoids associated with curves in surfaces~\cite{Debord2001,Gualtieri2014a}, or more precisely a groupoid integrating the Lie algebroid $\cT{\X}(-\D)$ of vector fields vanishing on $\D$, via blowing up.  Namely, a length-one subscheme is simply a point of $\X$ with the reduced scheme structure, so that $\Hilb[1]{\X}=\X$, and hence by \autoref{prop:representable}, we have a natural birational map
\[
(s,t) : \GHilb[1]{\X}{\D} \to \X\times\X.
\]
A pair of points in $\Hilb[1]{\X}$ are $\D$-equivalent if and only one of two things occurs: either they both lie in $\X\setminus \D$, or they are the same point of $\D$.  It follows that the isotropy groups of points $p \in \D$ are two-dimensional, and so the preimage of the diagonal embedding of $\D$ under the map $(s,t)$ is a hypersurface in $\GHilb[1]{\X}{\D}$. Hence $(s,t)$ factors through the blowup of $\X\times \X$ along the diagonal copy of $\D$.  One can check (e.g.~using \autoref{thm:triang_chart_Hilb_y_affk} below)  that this identifies $\GHilb[1]{\X}{\D}$ with the complement of the strict transforms of $\D \times \X$ and $\X \times \D$ in the blowup.  It would be interesting to give a similarly explicit description of the birational map $\GHilb{\X}{\D}\to\X\times\X$ when $n >1$ or $\D$ is singular.
\end{example}

\subsection{Stabilizer groups}
\label{sec:stabilizers}

\begin{definition}\label{def:stab-Lie}
We denote by $\Stab{\Z}{\D}$ the stabilizer group algebra of $\Z \in \Hilb{\X}$ in the groupoid $\GHilb[n]{\X}{\D}$, i.e.~the group of homotopy classes of infinitesimal self $\D$-equivalences of $\Z$, and by $\stab{\Z}{\D}$ its Lie algebra.
\end{definition}
By \autoref{lem:D-equiv-equivs},  $\Stab{\Z}{\D}$ can be identified with the group of automorphisms of $i^*\cI_\Z$ that act as the identity on the maximal torsion-free quotient $i^*\cI_\Z^\tf$.  Correspondingly, $\stab{\Z}{\D}$ is identified with the endomorphisms of $i^*\cI_\Z$ that act by zero on $i^*\cI_\Z^\tf$.  Since the maximal torsion subsheaf  
\[
\tormod{\Z}{\D} := \cTor[1][\cO{\X}]{\cO{\Z},\cO{\D}} = \scoH[-1]{\cO{\Z\hcap\D}}
\]
of $i^*\cI_\Z$ is preserved by any endomorphism, we immediately deduce the following:

\begin{lemma}\label{lem:stab-Lie-exact}
There is a canonical exact sequence of groups
\[
\xymatrix{
0 \ar[r] & \Hom[\D]{i^*\cI_\Z^\tf, \tormod{\Z}{\D}} \ar[r] & \Stab{\Z}{\D} \ar[r] & \Aut[\D]{\tormod{\Z}{\D}} 
}
\]
where the group law on $\Hom[\D]{i^*\cI_\Z^\tf, \tormod{\Z}{\D}}$ is addition.  Correspondingly we have an exact sequence of Lie algebras
\[
\xymatrix{
0 \ar[r] & \Hom[\D]{i^*\cI_\Z^\tf, \tormod{\Z}{\D}} \ar[r] & \stab{\Z}{\D} \ar[r] & \End[\D]{\tormod{\Z}{\D}} 
}
\]
where $\Hom[\D]{i^*\cI_\Z^\tf, \tormod{\Z}{\D}}$ is abelian.
\end{lemma}

\subsection{Syzygy matrices and orbit data}

We now turn to the classification of the orbits of $\GHilb{\X}{\D}$, i.e.~the enumeration of the $\D$-equivalence classes.  In light of \eqref{lem:D-equiv-module}, this is equivalent to the classification of $\cO{\D}$-modules $\cF$ such that $\cF \cong i^*\cI_\Z$ for some $\Z \in \Hilb{\X}$.  

In order to characterize such modules, note that since $\X$ is smooth of dimension two the stalk of $\cI_\Z$ at any point $p \in \X$ is an $\cO{\X,p}$-module of projective dimension at most two; indeed, if $p \notin \Z$ then $\cI_{\Z,p} = \cO{\X,p}$, while if $p \in \Z$, the classical Hilbert--Burch theorem states that the minimal resolution has the form
\begin{align}
\xymatrix{
0 \ar[r] & \cO{\X,p}^{\oplus k} \ar[r]^-{\syz} & \cO{\X}^{\oplus (k+1)} \ar[r]^-{\minsyz} & \cI_{\Z,p} \ar[r] & 0 
}\label{eq:min-res-X}
\end{align}
where $\syz$ is a $(k+1)\times k$-matrix of elements of the maximal ideal $\m_{\X,p} < \cO{\X,p}$, called the \defn{syzygy matrix}, and $\minsyz$ is the row vector formed from the maximal minors of $\syz$ with alternating signs, i.e.
\[
\minsyz = \begin{pmatrix} \det(\syz_0) & -\det(\syz_1) & \cdots & (-1)^k\det(\syz_k) \end{pmatrix}
\]
where $\syz_j$ is the $k\times k$ submatrix obtained by removing the $(j+1)$st row from $\syz$.  The matrix $\syz$ is determined by $\Z$ and $p$ up to multiplication on the left and right by elements in $\GL[k+1]{\cO{\X,p}}$ and $\GL[k]{\cO{\X,p}}$, respectively.

Applying the functor $Li^*$ to \eqref{eq:min-res-X}, we obtain a minimal resolution
\begin{align}
\xymatrix{
0 \ar[r] & \cO{\D,p}^{\oplus k} \ar[r]^-{\syz|_\D} & \cO{\D,p}^{\oplus (k+1)} \ar[r]^-{M(\syz|_\D)} & i^*\cI_{\Z,p} \ar[r] & 0 
}\label{eq:min-res-D}
\end{align}
for any $p \in \D$, so that the $\D$-equivalence class of $\Z$ near $p$ is controlled by the equivalence class of the matrix $\syz|_\D$.  For instance, $\syz|_\D$ can be used to construct an explicit model for $\cO{\Z\hcap \D}$ as a dg $\cO{\D}$-algebra.  Namely, by extending \eqref{eq:min-res-X} one step to the right we obtain a minimal resolution of $\cO{\Z,p}$, which carries a canonical dg algebra structure due to Herzog~\cite{Herzog1973} (see also \cite[Section 6]{Beck2014}), whose structure constants are determined by $\syz$.  Tensoring the resolution with $\cO{\D,p}$, we obtain an explicit model for $\cO{\Z\hcap \D,p}$ as an $\cO{\D,p}$-algebra.

 This motivates the following definition. Note that in this definition, there is a unique local orbit datum of size $k=0$, corresponding to the case in which the point $p \in \D$ does not lie in $\Z$ and $S|_\D$ is the zero map $0 \to \cO{\D,p}$, which we think of as a matrix of size $1 \times 0$.

\begin{definition}
A \defn{local orbit datum supported at $p \in \D$ of size $k \ge 0 $} is an element of the double quotient
\begin{align}
\GL[k+1]{\cO{\D,p}} \backslash \m_{\D,p}^{(k+1)\times k} / \GL[k]{\cO{\D,p}}, \label{eq:orb-datum}
\end{align}
which can be realized as the restriction to $\D$ of the syzygy matrix of a zero-dimensional subscheme $\Z \subset \X$.  We denote the set of local orbit data of size $k$ by $\orbdat[k]{\D}{p}$, and set
\[
\orbdat{\D}{p} := \bigsqcup_{k \ge 0} \orbdat[k]{\D}{p} 
\]
A \defn{(global) orbit datum} is an assignment $\phi : p \mapsto \phi(p)$ of a local orbit datum $\phi(p) \in \orbdat[k]{\D}{p}$ to every point $p \in \D$ such that there are only finitely many $p$ with $\phi(p)$ of positive size.
\end{definition}
Note that every global orbit datum can be realized by a zero-dimensional subscheme, simply by taking the disjoint union of the zero-dimensional subschemes at each $p$ where $\phi(p)$ has positive size.

If $\Z \in \Hilb{\X}$, we denote by
\[
[\Z]_p \in \orbdat{\D}{p}
\]
the image of the syzygy matrix of $\Z$, which is well-defined.  If $\phi$ is a global orbit datum, we define a subset
\[
\Hilb{\X}_{\phi} := \set{ \Z \in \Hilb{\X} }{ [\Z]_p= \phi(p) \textrm{ for all }p \in \D} 
\]
Then clearly $\Hilb{\X}_{\phi}$, if non-empty, is a $\D$-equivalence class, i.e.~an orbit of $\GHilb{\X}{\D}$, and every $\D$-equivalence class arises in this way.

\subsection{Classification of orbits}
\label{sec:orbits}
\subsubsection{Smooth points}

Let $p \in \D$ be a smooth point of $\D$ and suppose that $\Z \in \Hilb{\X}$.  Then the ring $\cO{\D,p}$ of analytic germs at $p$ is a principal ideal domain, isomorphic to the ring $\CC\{x\}$ of convergent power series in an coordinate $x$ centred at $p$. Then since  $i^*\cI_{\Z,p}$ is a finitely generated module whose free part has rank one, it must have the form
\begin{align}
i^*\cI_{\Z,p} &\cong \cO{\D,p} \oplus \bigoplus_{j=1}^k \cO{\D,p}/\m_{\D,p}^{\mu_j} \\
& \cong \CC\{x\} \oplus \bigoplus_{j=1}^k \CC[x]/x^{\mu_j}\label{eq:smooth-point-modules}
\end{align}
for a unique collection $\mu = (\mu_1 \ge \mu_2 \ge \cdots )$ of positive integers.  Correspondingly, the syzygy matrix can be put in the Smith normal form
\begin{align}
\syz|_\D \sim S_\mu(x) := \begin{pmatrix}
x^{\mu_1}& 0& 0 & \dots & 0 \\
0 & x^{\mu_2} & 0 & \dots & 0  \\
0 & 0 & x^{\mu_3} & \dots & 0  \\
\vdots & \vdots & \vdots & \ddots & \vdots \\
0 & 0 & 0 & \dots & x^{\mu_{k}} \\
0 & 0 & 0 & \dots & 0
\end{pmatrix}. \label{eq:smith-form}
\end{align}
Recall the following definition.
\begin{definition}
A \defn{Young diagram} (of \defn{length $k$}) is a decreasing sequence $\mu = (\mu_1\ge \cdots \ge \mu_k)$ of positive integers.  Its \defn{size} is its sum:
\[
|\mu| = \mu_1 + \cdots + \mu_k.
\]
By convention, we set $\mu_j = 0$ for $j > k$.
\end{definition}

We depict a Young diagram $\lambda$ by placing $\lambda_j$ boxes in the $j$-th row, starting the numbering of rows from $1$, e.g.
\[
(4,2,2,1) \ \ \leftrightarrow\ \  \ytableausetup{centertableaux,smalltableaux}
\begin{ytableau}
~ \\
&  \\
& \\
& & & 
\end{ytableau}
\]
is a Young diagram of length four and size nine.

Rephrasing the above, we have the following characterization of the local orbit data at $p$:
\begin{lemma}\label{lem:smooth-orbit-data}
If $p \in \D$ is a smooth point, the map sending a Young diagram $\mu$ to the equivalence class of the corresponding Smith normal form matrix gives a canonical bijection
\[
\orbdat[k]{\D}{p} \cong \{\textrm{Young diagrams of length k}\}.
\]
In particular, if $\D$ is smooth, then the global orbit data are in bijection with functions from $\D$ to the set of Young diagrams of arbitrary size.
\end{lemma}

We now describe the basic characteristics of the orbit datum corresponding to a Young diagram $\mu$.  In order to do so, we introduce the following auxiliary notion:

\begin{definition}\label{def:horizontally_convex_yd}
We say that a Young diagram $\lambda=(\lambda_1\ge \lambda_2\ge...)$ is \defn{horizontally convex}  if $\lambda_i-\lambda_{i+1} \ge \lambda_{i+1}-\lambda_{i+2}$, for every $i>0$.
\end{definition}
For instance, the following Young diagrams  are horizontally convex:
$$
\ytableausetup{centertableaux,smalltableaux}
\begin{ytableau}
\\
&\\
&&\\
\end{ytableau}
~~~
\begin{ytableau}
\\
& &\\
& & & & \\
\end{ytableau}
~~~
\begin{ytableau}
~&\\
 & & & & \\
& & & & & & & & \\
& & & & & & & & && & & &\\
\end{ytableau}
$$
whereas the following are \emph{not} horizontally convex:
$$
\begin{ytableau}
\\
\\
\end{ytableau}
~~~
\begin{ytableau}
~&\\
&\\
\end{ytableau}
~~~
\begin{ytableau}
~&\\
& &\\
\end{ytableau}
~~~
\begin{ytableau}
\\
~&\\
& \\
\end{ytableau}
~~~
\begin{ytableau}
~&\\
& & & & \\
& & & & & & \\
\end{ytableau}
$$

For a Young diagram $\mu$, let $\hc{\mu}$ be the unique Young diagram $\lambda$ such that $\lambda_j -\lambda_{j+1} = \mu_j$ for all $j > 0$, or  equivalently $\lambda_j = \sum_{l =1}^j \mu_l$.  Then $\hc{\mu}$ is horizontally convex, and evidently every horizontally convex diagram can be obtained in this way. For example,
\begin{align}
\mu ~~=~~~~
\begin{ytableau}*(white!50!blue) \\
 *(white!50!blue)& *(white!50!blue)\\
 *(white!50!blue)&  *(white!50!blue)& *(white!50!blue)&*(white!50!blue)  \\
 *(white!50!blue)& *(white!50!blue) &*(white!50!blue) & *(white!50!blue)&*(white!50!blue)\\
\end{ytableau}
\ \ \ \ \longmapsto\ \ \ \ 
 \hc{\mu} ~~=~~~~
\begin{ytableau}
*(white!50!blue)\\
 & *(white!50!blue)& *(white!50!blue)\\
& & &*(white!50!blue)&*(white!50!blue)&*(white!50!blue)&*(white!50!blue) \\
& & & & & & &*(white!50!blue) &*(white!50!blue)&*(white!50!blue)&*(white!50!blue)&*(white!50!blue) \\
\end{ytableau} \label{eq:hc}
\end{align}
Note that the number of boxes of $\hc{\mu}$ is given by
\begin{align}
|\hc{\mu}| = \sum_{j \ge 1} j \mu_j, \label{eq:size-of-hc}
\end{align}
which will be useful below.

Given a Young diagram $\mu$ and a smooth point $p \in \D$, we define a zero-dimensional subscheme
\[
\Z_p^\mu \subset \X
\]
supported at $p$ as follows.  First, choose coordinates $(x,y)$ centred at $p$, such that $\D$ is given locally by the equation $y=0$, and let $\lambda = \hc{\mu}$.  Then set
\begin{equation}\label{eq:D-convex_subscheme}
\Z_p^\mu := \textrm{the vanishing locus of the monomials } x^jy^l, \qquad j \ge \lambda_{l+1}
\end{equation}
Using the horizontal convexity of $\lambda$, it is straightforward to verify that this definition is independent of the choice of coordinates $(x,y)$. 

Evidently the algebra $\cO{\Z_p^\mu} \cong \CC[x,y]/(x^jy^{l})_{j \ge \lambda_{l+1}}$ has a basis given by the monomials $x^jy^l$ where $j < \lambda_{l+1}$ for $l=0,\ldots,k-1$.
The following is therefore immediate:
\begin{lemma}\label{lem:length-young-diagram-monomoial} 
The algebra $\cO{\Z_p^\mu}$ has a basis indexed by the boxes of the Young diagram $\hc{\mu}$.  In particular, the length of the scheme $\Z_p^\mu$ is given by
\[
\# \Z_p^\mu = |\hc{\mu}|.
\]
\end{lemma}

\begin{remark}\label{rem:conv-monomial-limit}
$\Z_p^\mu$ can be defined in a coordinate-free fashion as follows.  Let $\mu^T$ be the transpose of $\mu$, defined by $\mu_j^T = \#\set{l > 0}{\mu_l \ge j}$, and let $k^T = \mu_1$ be its length.  For a collection $p_1,\ldots,p_{k^T}$ of pairwise distinct points in $\D$, consider the subscheme given by the disjoint union of $\mu^T_j$-th order neighbourhoods of $p_j$ for $1 \le j \le k^T$.  Taking the limit as the points $p_1,\ldots,p_{k^T} \to p$ we obtain the subscheme $\Z^\mu_p$.

For example, consider the Young diagrams
\[
\mu :=~ \begin{ytableau}
~& \\
& \\
&&&& \\
&&&&&
\end{ytableau} \qquad\qquad
\lambda := \hc{\mu}=~
\begin{ytableau}
~&\\
 & & & \\
& & & & & & & & \\
& & & & & & & & && & & & &\\
\end{ytableau}
\qquad\qquad
\mu^T =~ \begin{ytableau}
~ \\
& \\
& \\
& \\
&&&\\
&&&
\end{ytableau}
\]
Then $\Z_p^\mu$ is obtained as the limit of the  schemes defined by the intersection of the ideals $\m_{\X,p_1}^4 , \m_{\X,p_2}^4 , \m_{\X,p_3}^2 , \m_{\X,p_4}^2 , \m_{\X,p_5}^2 , \m_{\X,p_6}$, when all points $p_i\in \D$ tend to $p$ while remaining distinct. Pictorially, this corresponds to the fact that the Young diagram $\lambda$ can be obtained taking the Young diagrams
$$
\begin{ytableau}
~  \\
& \\
& & \\
& & & \\
\end{ytableau}
~~~
\begin{ytableau}
~  \\
& \\
& & \\
& & & \\
\end{ytableau}
~~~
\begin{ytableau}
\none\\
\none\\
~  \\
& \\
\end{ytableau}
~~~
\begin{ytableau}
\none\\
\none\\
~  \\
& \\
\end{ytableau}
~~~
\begin{ytableau}
\none\\
\none\\
~  \\
& \\
\end{ytableau}
~~~
\begin{ytableau}
\none\\
\none\\
\none\\
~  
\end{ytableau}
$$
and piling their boxes onto one another horizontally.
\end{remark}

The importance of the subscheme $\Z_p^\mu$ is that it gives a canonical representative of the corresponding orbit datum:
\begin{proposition}
Let $\mu$ be a Young diagram.  Then $[\Z_p^\mu] = \mu \in \orbdat{\D}{p}$.
\end{proposition}

\begin{proof}
  Let $\Z = \Z_p^\mu$.  In light of the above, $i^*\cI_{\Z,p}$ is determined up to isomorphism by its torsion subsheaf, which is isomorphic as an $\cO{\D}$-module to
  $\cTor[1]{\cO{\Z},\cO{\D}}$.   In local coordinates $(x,y)$ such that $\D$ is the vanishing locus of $y$, we can use the Koszul resolution of $\cO{\D}$ to compute
$$
\cTor[1]{\cO{\Z},\cO{\D}} \cong \ker \Big(
\cO{\Z} \xrightarrow{~~y \cdot~~} \cO{\Z}
\Big)
$$
Using the basis of $\cO{\Z}$ indexed by the boxes of $\lambda := \hc{\mu}$ as above, the operator $y$ corresponds to a translation upwards by one step.  Therefore $\cTor[1]{\cO{\Z},\cO{\D}}$ has a $\CC$-basis indexed by the boxes of $\lambda$ that sit at the top of the columns, and these are in bijection with the boxes of $\mu$, as illustrated in \eqref{eq:hc}.  Moreover $x$ restricts to a coordinate on $\D$, identifying $\cO{\D,p} \cong \CC\{x\}$.  The action of $x$ on the basis corresponds to translation rightwards by one step in the Young diagram, and in this way we deduce that $\cTor[1]{\cO{\X},\cO{\D}}$ is isomorphic to the direct sum of the modules $\CC\{x\}/(x^{\mu_j})$ for $j =1,\ldots, k$, as desired.
\end{proof}

\begin{proposition}\label{prop:contains-hor-convex}
  If $\Z \in \Hilb{\X}$ is any element whose orbit datum at $p$ is the Young diagram $\mu$, then $\Z_p^\mu$ is a subscheme of $\Z$.  Moreover the orbit of $\Z$ contains as a Zariski open dense subset the locus of elements of the form $\Z' \sqcup \Z_p^{\mu}$ such that $p \notin \Z'$.
\end{proposition}
\begin{proof}
Choose local coordinates $(x,y)$  centred at $p$ for which $\D = \{y=0\}$.  By assumption, the syzygy matrix for $\Z$ is equivalent to one of the form
\[
S_\Z = S_\mu(x) + y M = \sum_{j=1}^k x^{\mu_{j}} E_j + y M
\]
where $E_j$ is whose $(j,j)$-entry is 1 and whose other entries are zero, $M \in \cO{\X,p}^{(k+1)\times k}$ and $S_\mu(x)$ is the Smith normal form matrix~\eqref{eq:smith-form}.  Moreover $\Z$ is given locally by the vanishing of the minors of $S_\Z$.  Using the skew-multilinearity of the minors of a matrix, we deduce that each minor is a sum of terms the form $x^{\mu_{j_1}} \cdots x^{\mu_{j_l}} y^{k-l} \det(M_{j_1,\ldots,j_k})$ for suitable matrices $M_{j_1,\ldots,j_k} \in \cO{\X,p}^{k\times k}$, where the indices $j_l,\ldots,j_l$ are pairwise distinct. Since $\mu_k \le \mu_{k-1}\le \ldots$, each such term is divisible by $x^{\mu_{k}+\cdots+\mu_{k-l+1}}y^{k-l} = x^{\hc{\mu}_{l-1}}y^{k-l}$ and hence $\cI_\Z$ is contained in the ideal defining $\Z_p^\mu$, or equivalently $\Z_p^\mu \subset \Z$, as desired.

For the second statement, first note that the Zariski open
property follows by upper semicontinuity of the multiplicity of the point $p$; this upper semicontinuity follows from the continuity of the Hilbert--Chow morphism, and the fact that multiplicity is clearly upper semicontinuous on the base.  So we only need to prove that the given locus is dense.
The problem localizes around the points of $\Z$, so we may assume without loss of generality that $\X = \CC^2$; in this case $\D$ is anticanonical so by \autoref{cor:sympl_leaves_general} below, the orbits are disjoint unions of symplectic leaves of one of Bottacin's Poisson structures on $\Hilb{\X}$.  Furthermore, we may assume that  $\Z$ is supported at a point $p\in \D$, with $\Z \neq \Z_p^\mu$. Hence $\# \Z > \# \Z_{p}^\mu = |\hc{\mu}|$, so by \autoref{lem:length-young-diagram-monomoial} above and \autoref{lem:smooth-stablizer} below, the  orbit of  $\Z$ has positive dimension.  Since the monomial ideals give isolated point in $\Hilb{\X}$, we may assume in addition that the ideal defining $\Z$ is not generated by monomials.  Then the vector field on $\Hilb{\X}$ induced by the Euler vector field $E=\sum_{i=1}^n y_i\partial_{y_i}$ generating the standard $\mathbb{C}^*$ action in coordinates is non-zero at $\Z$ (e.g. see \cite[Proposition 7.4]{Nakajima1999}). Pick a locally defined function $g\in\mathcal{O}_{\Hilb{\X}}$ such that $E(g) = 1$. Then we have $\{g,h\}=1$, here $h$ is the locally defined function on $\Hilb{\X}$ sending $\Z$ to the sum of the $x$-coordinates of its support. Therefore, the Hamiltonian vector field of $g$ pushes $\Z$ away from the locus where $\sum_{i=1}^n x_i=0$. In particular, it pushes $\Z$ away from the set of elements of $(\mathbb{C}^2)^{[n]}$ supported at $p$.  Since the flow cannot change the intersection with $\D$, we deduce that $\Z$ is equivalent to a scheme of the form $\Z' \sqcup \Z''$ where $\Z' \neq \varnothing$ is disjoint from $p$ and $\Z'' \supset \Z_p^\mu$.  The result follows by downward induction on the length of $\Z$.
\end{proof}
\begin{remark}
  Note in the above proof, we cite \autoref{cor:sympl_leaves_general} below. There is no circular logic because the above proposition is not used after the present subsection. Moreover \autoref{cor:sympl_leaves_general} (and all of  \autoref{ss:symp-groupoid-leaves}) do not require anything from the present \autoref{sec:orbits}.
\end{remark}

Note that in light of \eqref{eq:smooth-point-modules} we have a (non-canonical) splitting $i^*\cI_{\Z,p} \cong i^*\cI_{\Z,p}^\tf \oplus \tormod{\Z}{\D}$, where
\[
i^*\cI_{\Z,p}^\tf \cong \cO{\D,p} \cong \CC\{x\} \qquad \tormod{\Z}{\D,p} \cong \bigoplus_{j=1}^k \cO{\D,p}/\m_{\D,p}^{\mu_j} \cong \bigoplus_{j=1}^k \CC[x]/(x^{\mu_j})
\]
are the torsion-free and torsion parts, respectively.  It follows immediately that the natural map $\Stab{\Z}{\D} \to \Aut[\D]{\tormod{\Z}{\D}}$ is a split surjection, with kernel $\Hom{i^*\cI_\Z^\tf,\tormod{\Z}{\D}} \cong \coH[0]{\tormod{\Z}{\D}}$, giving the following description of the stabilizer group:

\begin{lemma}\label{lem:smooth-stablizer}
Suppose $\Z \cap \D$ is supported in the smooth locus of $\D$. Then the stabilizer group of $\Z$ in $\GHilb{\X}{\D}$ is given by 
\[
\Stab{\Z}{\D} \cong \Aut[\D]{\tormod{\Z}{\D}} \ltimes \coH[0]{\tormod{\Z}{\D}}
\] 
where $\coH[0]{\tormod{\Z}{\D}}$ is viewed as an abelian group under addition.  In particular, the orbit of $\Z$ has codimension
\[
\codim(\GHilb{\X}{\D} \cdot \Z) = 2\sum_{p \in \Z} |\hc{\mu(p)}|
\]
where $\mu(p)$ is the Young diagram representing the orbit datum $[\Z]_p$ for all $p \in \D$.
\end{lemma}

\begin{proof}
It remains to establish the formula for the codimension of the orbit, or equivalently the dimension of the Lie algebra $\stab{\Z}{\D}$ of $\Stab{\Z}{\D}$.  But the latter is the sum of the dimensions of the vector spaces
\[
\coH[0]{\tormod{\Z}{\D}} \cong \bigoplus_{j=1}^k \CC[x]/(x^{\mu_j}) \qquad \textrm{and} \qquad 
\End[\D]{\tormod{\Z}{\D}} \cong \bigoplus_{j,l=1}^k \Hom[{\CC[x]}]{\CC[x]/(x^{\mu_j}),\CC[x]/(x^{\mu_l})}.
\]
Note that $\dim \CC[x]/(x^{\mu_j}) = \mu_j$, while $\dim  \Hom[{\CC[x]}]{\CC[x]/(x^{\mu_j}),\CC[x]/(x^{\mu_l})} = \min\{\mu_j,\mu_l\}$.  Therefore 
$$
\dim \Stab{\G}{\Z} = \sum_{j \ge1} \mu_j + \sum_{j,l\ge 1} \min\{\mu_j, \mu_l\} =  \sum_{j\ge 1} \mu_j + \sum_{j\ge 1} (2j-1) \mu_j = 2 \sum_{j\ge 1} j \mu_j =  2 |\hc{\mu}|.
$$
by \eqref{eq:size-of-hc}. 
\end{proof}

Combined with \autoref{lem:length-young-diagram-monomoial}, we deduce the following:
\begin{corollary}
  The zero-dimensional orbits of $\GHilb{\X}{\D}$ are exactly the elements of the form $\Z = \bigsqcup_{p\in\D}  \Z_p^{\mu(p)}$, such that $\sum_{p \in \D}|\hc{\mu(p)}| = n$.
\end{corollary}

Combining these results, we have obtained the following description of the orbits in the case where $\D$ is smooth:
\begin{theorem}\label{thm:sympl_leaves_Hilb_y}
Suppose that $\D \subset \X$ is smooth.  Then the orbits of $\GHilb{\X}{\D}$ are in bijection with functions $\mu$ from $\D$ to the set of all Young diagrams, such that
\[
\sum_{p \in \D} |\hc{\mu(p)}| \le n.
\]
Moreover, for any such $\mu$, the set $\set{\Z' \sqcup \bigsqcup_{p \in \Z} \Z_p^{\mu(p)}}{\Z' \subset \X\setminus \D} \cong \Hilb[n-\sum_p|\hc{\mu(p}|]{(\X\setminus \D)}$ is open and dense in the corresponding orbit $\Hilb{\X}_\mu$.  In particular, if $\X$ is connected, then all orbits are connected.
\end{theorem}

\begin{remark}
In addition, one can show that the orbit $\Hilb{\X}_\mu$ lies in the closure of the orbit $\Hilb{\X}_{\tmu}$ if and only if $\tmu \ge \mu$ in the dominance order on Young diagrams.  However the proof uses more refined information about the local structure of $\Hilb{\X}$, so we postpone it; see \autoref{cor:Hilb_y_sympl_leaves_order}.
\end{remark}

\subsubsection{Nodes}

We now briefly describe the classification of the orbit data at a nodal singularity $p \in \D$.

Choose coordinates $(x,y)$ centred at $p$, so that $\D = \{xy=0\}$, giving an isomorphism
\[
\cO{\D,p} \cong \CC\{x,y\}/(xy).
\]
The classification of finitely generated modules over this ring is known: the torsion modules were classified in \cite{GelPon68} (see also \cite{BGS09}), and the classification of all finitely generated modules over $\mathbb{C}[x,y]/(xy)$ was done in \cite{Levy1,Levy2}, based on earlier works \cite{NazRoi69,NRSB72}. In particular, every finitely generated module decomposes as a direct sum of indecomposable ones, and this decomposition can be determined algorithmically~\cite{LaubStur96}.  Furthermore, the minimal resolutions of the indecomposables are presented in \cite[Appendix A]{Burban2004}.  Of these, only two families have projective dimension at most one, namely  those listed as (1) and (3) on p.~224--225 of \emph{op.~cit.}:
\begin{itemize}
\item  A ``continuous series'' of torsion modules, expressed as the cokernels of the square matrices
\[
M(\textbf{d},k,u) := \begin{pmatrix}
x^{\nu_1} I_k & 0 & 0 &\cdots & y^{\mu_l} J_k(u) \\
y^{\mu_1}I_k & x^{\nu_2}I_k &0 & \cdots& 0 \\
0 & y^{\mu_2}I_k & x^{\nu_3}I_k &\cdots & 0 \\
\vdots & \vdots & \vdots &\ddots & \vdots \\
0 & 0 &\cdots & y^{\mu_{l-1}}I_k & x^{\nu_l} I_k
\end{pmatrix}
\]
of size $k\cdot l$, where $\textbf{d} = (\nu_1,\mu_1),(\nu_2,\mu_2),\ldots,(\nu_l,\mu_l)$ is a non-periodic sequence of pairs of positive integers, $I_k$ denotes the identity matrix of size $k$ and $J_k(u)$ is a maximal Jordan block of size $k$ with eigenvalue $u \neq 0$.  We include here the degenerate case $l=1$, in which case $M(\mathbf{d},k,u) = x^{\nu_1} I_k-y^{\mu_1} J_k(u)$.
\item A ``discrete series'' of rank-one modules, expressed as the cokernels of the $(k+1)\times k$ matrices
\[
\mathscr{D}(\mathbf{d}) := \begin{pmatrix}
x^{\nu_1} & 0 & 0 & \cdots & 0 \\
y^{\mu_1} & x^{\nu_2} & 0 & \cdots & 0 \\
0 & y^{\mu_2} & x^{\nu_3} & \cdots & 0 \\
0 & 0 & y^{\mu_3} & \ddots & \vdots \\
\vdots & \vdots&\vdots & \ddots & x^{\nu_k} \\
0 & 0 & 0 & \cdots & y^{\mu_k}
\end{pmatrix}
\]
where $\mathbf{d} = (\mu_1,\nu_1),\cdots, (\mu_k,\nu_k)$ is a sequence of pairs of positive integers. We include here the degenerate case $k=0$, which corresponds to the trivial module.
\end{itemize}

Since for any $\Z \in \Hilb{\X}$, the module $i^*\cI_{\Z,p}$ has rank one and projective dimension at most one, we have the following classification of the orbit data:
\begin{proposition}\label{prop:orbit-data-node}
If $p$ is a node, then each local orbit datum is presented as the direct sum of a single matrix of the form $\mathscr{D}(\mathbf{d})$ and at most finitely many matrices of the form $M(\mathbf{d},k,u)$. This presentation is unique up to permutation of the summands.
\end{proposition}

\begin{remark}
By choosing suitable lifts of the direct sums of the matrices $\mathscr{D}(\textbf{d})$ and $M(\mathbf{d},k,u)$ to matrices valued in $\cO{\X,p}$, one can show that all direct sums as in \autoref{prop:orbit-data-node} are realized as orbit data for sufficiently large $n$.  It would be interesting to know the minimal $n$ for which a given direct sum can occur.
\end{remark}
\begin{remark}
It is straightforward to check that the indecomposable rank one modules in the discrete series above are determined up to isomorphism by their torsion submodules, which have infinite projective dimension and form the discrete series described in item (2) on p.~225 of~\cite{Burban2004}. Hence, if the singularities of $\D$ are at worst nodal, the orbit type of an element $\Z \in \Hilb{\X}$ is completely determined by the isomorphism class of the torsion submodule $\tormod{\Z}{\D}$ of $\cI_\Z$. 
\end{remark}

\section{Symplectic and Poisson structures}
\label{sec:pois}

\subsection{Bottacin's Poisson structure}
We now consider the particular case in which the divisor $\D\subset \X$ is anticanonical, and we fix a section
\[
\ps \in \coH[0]{\acan{\X}}
\]
vanishing on  $\D$.  Thus $\ps$ is a Poisson structure on $\X$ having the open set $\X\setminus \D$ as a two-dimensional symplectic leaf, and the points of $\D$ as zero-dimensional symplectic leaves.   In \cite{Bottacin1998}, Bottacin shows that for every $n \ge 0$, the induced Poisson structure $\symps$ on the symmetric power $\sympow{\X}$ lifts along the Hilbert--Chow morphism to a Poisson structure on the Hilbert scheme $\Hilb[n]{\X}$, and that the corresponding bivector
\[
\hilbps \in \coH[0]{\wedge^2 \cT{\Hilb[n]{\X}}}
\]
 has the following description: using the various identifications \eqref{eq:Hilb-tangent} of the tangent space at $\Z \in \Hilb[n]{\X}$, and the corresponding Serre dual descriptions of the cotangent space, the anchor map $(\hilbps)^\sharp : \ctb[\Z]{\Hilb[n]{\X}} \to \tb[\Z]{\Hilb[n]{\X}}$ is determined by the following commutative diagram:
\[
\xymatrix{
\ctb[\Z]{\Hilb[n]{\X}} \ar[r]^-{\sim} \ar[d]^-{(\hilbps)^\sharp} & \Ext[1][\X]{\cI_\Z,\cI_\Z\otimes\can{\X}}_0 \ar[d]^-{\ps} \ar[r]^-{\sim}& \Hom[\X]{\cI_{\Z},\cO{\Z}\otimes\can{\X}} \ar[d] ^-{\sigma} \\
\tb[\Z]{\Hilb[n]{\X}}  & \Ext[1][\X]{\cI_\Z,\cI_\Z}_0 \ar[l]^-{\sim} &  \Hom[\X]{\cI_{\Z},\cO{\Z}} \ar[l]^-{\sim}
}
\]

\subsection{Symplectic groupoid and symplectic leaves}\label{ss:symp-groupoid-leaves}

We now sketch an alternative approach to the construction of Bottacin's Poisson structure $\hilbps$, based on the formalism of shifted symplectic structures from~\cite{Pantev2013}, and the now standard correspondence between Poisson structures, symplectic groupoids and shifted Lagrangian morphisms. 
  This conceptual approach has the advantage of immediately yielding a description of the symplectic groupoid and symplectic leaves of $\hilbps$, and is indeed how we arrived at the definition of the groupoid $\GHilb{\X}{\D}$.  However, \emph{a posteriori} one can construct the symplectic structure on the groupoid without invoking the theory of shifted symplectic structures; see \autoref{rem:classical-construction} below.

   We will assume here that $\X$ is algebraic, since the cited references make this assumption.  However, the constructions are effectively local in $\X$, so that the result holds also in the non-algebraic setting with essentially the same arguments.

First, we observe that if $\D \subset \X$ is an anticanonical divisor, the map $Li^* : \Perf{\X}\to\Perf{\D}$ of derived stacks carries a one-shifted Lagrangian structure in a neighbourhood of any perfect complex on $\X$ with proper support.  If $\X$ itself is proper, this can be obtained by viewing the anticanonical section $\ps$ as a relative orientation on the inclusion $\D \to \X$, in the sense of \cite{Calaque2015a}, and combining the main result of \emph{op.~cit.} with the 2-shifted symplectic structure on the stack $\Perf{-}$ of perfect complexes from \cite{Pantev2013}.  More generally, one can invoke the results of~\cite{Brav2021}.  We may then pull this Lagrangian structure back along the square
\[
\xymatrix{
\Perfo{\X}\ar[r] \ar[d]& \Perf{\X}\ar[d] \\
\Perfo{\D}\ar[r] & \Perf{\D},
}
\]
thus obtaining a canonical isotropic structure on the map $\Perfo{\X}\to \Perfo{\D}$ over the locus of complexes with proper support.
\begin{lemma}
The induced isotropic structure on $\Perfo{\X}\to\Perfo{\D}$ is non-degenerate, i.e.~a Lagrangian structure.
\end{lemma}
\begin{proof}
It suffices to check that the induced pairing on tangent complexes is nondegenerate, but this follows immediately from the definition of the pairing in terms of Serre duality and the Illusie trace, since the decomposition $\cREnd[\X]{\cI_Z,\cI_Z} \cong \cREnd[\X]{\cI_\Z,\cI_\Z}_0 \oplus \cO{\X}$ is orthogonal with respect to the trace pairing.
\end{proof}

Restricting to the open substack $\Hilb[n]{\X}\subset \Perfo{\X}$, we deduce that the map $\Hilb[n]{\X} \to \Perfo{\D}$ is one-shifted Lagrangian, so that our groupoid
\[
\GHilb[n]{\X}{\D} \cong \Hilb[n]{\X} \fibprod_{\Perfo{\D}} \Hilb[n]{\X}
\]
carries a zero-shifted symplectic structure by the Lagrangian intersection theorem~\cite[Theorem 2.9]{Pantev2013}.  But by \autoref{prop:representable}, the derived stack $\GHilb[n]{\X}{\D}$ is a smooth classical scheme, so by \cite[p.~297-298]{Pantev2013} this 0-shifted symplectic structure is an ordinary symplectic structure in the classical sense.  Moreover, by definition, this symplectic structure on $\GHilb[n]{\X}{\D}$ is the difference of the pullbacks of the isotropic structure on $\Hilb[n]{\X}$ along the two maps $s,t : \GHilb[n]{\X}{\D} \to \Hilb[n]{\X}$.  It is therefore automatically multiplicative, so that $\GHilb[n]{\X}{\D}$ has the structure of a symplectic groupoid. As for any symplectic groupoid, the embedding of the identity arrows $\Hilb{\X}\hookrightarrow \GHilb{\X}{\D}$ is Lagrangian, giving an identification $\ctc{\Hilb{\X}}|_\Z \cong \tc s|_{(\Z,\Z,\id)}$. The induced Poisson bivector on $\Hilb{\X}$ is the composition of this isomorphism with the differential $\tc s |_{(\Z,\Z,\id)} \to\tc{\Hilb{\X}}|_\Z$ of the target map.  Considering the definition of the symplectic structure in terms of Serre duality and using the identifications of the (relative) tangent spaces from the proof of \autoref{prop:representable}, we deduce that this Poisson structure is exactly the one defined by Bottacin.

In summary, we have the following:
\begin{theorem}\label{thm:symplectic-groupoid}
Let $(\X,\ps)$ be a Poisson surface with degeneracy curve $\D = \ps^{-1}(0)\subset \X$.  Then $\GHilb[n]{\X}{\D}$ carries a canonical symplectic structure, making it a symplectic groupoid over $\Hilb[n]{\X}$ that integrates Bottacin's Poisson structure $\hilbps$ on $\Hilb[n]{\X}$.
\end{theorem}

\begin{remark}\label{rem:classical-construction}
Note that the symplectic structure on $\GHilb{\X}{\D}$ constructed above is the unique two-form for which the map $(s,t) : \GHilb{\X}{\D} \to \Hilb{\X}\times\Hilb{\X}$ is a symplectomorphism over the open dense set $\Hilb{(\X\setminus \D)} \times \Hilb{(\X\setminus \D)}$, where the latter is equipped with the symplectic form $\Omega := pr_1^*\omega-pr_2^*\omega$ obtained by pulling back the symplectic form $\omega$ on $\Hilb{\X\setminus \D}$ along the two projections $pr_1,pr_2$.  Thus, to prove that $\GHilb{\X}{\D}$ is symplectic, one simply needs to check that the pullback of $\Omega$ to $\GHilb{\X}{\D}$ has no poles, and similarly that the pullback of $\Omega^n$ has no zeroes.

These properties can be verified directly in at least two ways, giving a construction of the symplectic form that does not rely on the theory of shifted symplectic structures.  The first way is to explicitly write down the induced nondegenerate pairing on the tangent space at any point of $\GHilb{\X}{\D}$ using Grothendieck duality for coherent sheaves, and see directly that it agrees with the pullback of $\Omega$.   The second way begins with the observation that the desired properties are local and can be checked in codimension one.  Thus, using Weinstein's splitting theorem, one can reduce the problem to the case in which $n=1$, and the reduced curve underlying $\D$ is smooth.  In this case, the desired properties can be obtained by direct calculation using the explicit description of $\GHilb[1]{\X}{\D} \to \X\times \X$ in terms of blowups from \autoref{ex:blowup-construction}, in the spirit of \cite[Section 2.4]{Gualtieri2014a}, which treats the reduced case.
\end{remark}

Since the connected components of the orbits of any symplectic groupoid are exactly the symplectic leaves of the corresponding Poisson structure, and moreover the orbits of any smooth algebraic groupoid are locally closed algebraic subvarieties, \autoref{thm:symplectic-groupoid} has the following immediate consequences.
\begin{corollary}\label{cor:sympl_leaves_general}
The symplectic leaves of $(\Hilb{\X},\hilbps)$ are the connected components of the $\D$-equivalence classes. In particular, they are locally closed in the Zariski topology when $(\X,\ps)$ is algebraic, and if $\D$ is smooth they are in bijection with Young-diagram-valued functions on $\D$ as in \autoref{thm:sympl_leaves_Hilb_y}.
\end{corollary}

Note that the stabilizer groups in $\GHilb[n]{\X}{\D}$ are exactly the self $\D$-equivalences, and meanwhile, as for any symplectic groupoid, the Lie algebras of these groups are exactly the conormal Lie algebras of the Poisson structure (i.e.~the linearization of the Poisson structure in the direction transverse to the symplectic leaves) so we have the following:
\begin{corollary}
If $\Z \in \Hilb[n]{\X}$ is any point, then its conormal Lie algebra
\[
\ker\rbrac{ (\hilbps)^\sharp : \ctb[\Z]{\Hilb[n]{\X}} \to \tb[\Z]{\Hilb[n]{\X}}}
\]
is canonically identified with the Lie algebra $\stab{\Z}{\D}$ of infinitesimal self-$\D$-equivalences defined in \autoref{def:stab-Lie} and described in \autoref{lem:stab-Lie-exact}.
\end{corollary}

\subsection{Characteristic leaves}
\label{sec:char-leaves}
Of particular interest to us are the symplectic leaves that are ``characteristic'' in the sense of \cite{Matviichuk2020}; let us recall the definition.  First, recall from~\cite{Brylinski1999,Polishchuk1997,Weinstein1997} the notion of the modular vector field of a Poisson structure $\ps$ on a manifold $\W$ with respect to a volume form $\mu$: it is the derivation of $\cO{\W}$ sending a function $f$ to the $\mu$-divergence of the Hamiltonian vector field of $f$.  The modular vector field is an infinitesimal symmetry of the Poisson structure, and is independent of $\mu$ up to the addition of a Hamiltonian vector field, so that its projection to the normal space of any symplectic leaf is unambiguously defined.  We then have the following.
\begin{definition}[{\cite{Matviichuk2020}}]
A symplectic leaf of a Poisson manifold is \defn{characteristic} if it is preserved by the flow of the modular vector field, i.e.~the projection of the modular vector field to the normal bundle of the leaf is identically zero.
\end{definition}

Now suppose that $(\X,\ps)$ is a Poisson surface, and let $\D\subset \X$ be the anticanonical divisor on which $\ps$ vanishes.  Then the restriction of the modular vector field to $\D$ is independent of the choice of volume form, giving a canonically defined global vector field
\[
\zeta \in \coH[0]{\cT{\D}}.
\]
The characteristic leaves in the Hilbert scheme then have the following description.
\begin{proposition}\label{prop:char-leaves}
The symplectic leaf of $\hilbps$ through a point $\Z \in \Hilb[n]{\X}$ is characteristic if and only if the derived subscheme $\Z \hcap \D \subset \D$ is preserved (up to equivalence) by the flow of $\zeta$. 
\end{proposition}

\begin{proof}
The problem is local in $\X$ around $\Z$, so we may assume without loss of generality that that $\X$ admits a global volume form $\mu$ and let $\tilde \zeta$ be the corresponding modular vector field, so that $\tilde \zeta |_\D = \zeta$ by definition of $\zeta$.  Consider the product $\X^n$ and let $p_i : \X^n \to \X$ be the $i$-th projection.  For a tensor $\xi$ on $\X$ let $\xi^n = p_1^*\xi + \cdots + p_n^*\xi$ denote its symmetric lift to $\X^n$.  Then $\ps^n$, $\mu^n$ and $\tilde \zeta^n$ are, respectively, a Poisson structure, a volume form, and the corresponding modular vector field on $\X^n$.  These are invariant under permutation, and hence descend to the symmetric power $\sympow{\X}$, giving a Poisson structure  $\symps$, a trivialization $\sympow{\mu}$ of the canonical sheaf, and a vector field $\sympow{\tilde \zeta}$, which agrees with the modular vector field over the smooth locus of $\sympow{\X}$.  We claim that these structures lift to corresponding structures on the Hilbert scheme under the Hilbert--Chow morphism $\Hilb[n]{\X} \to \sympow{\X}$.  Indeed, $\sympow{\ps}$ lifts to Bottacin's structure $\hilbps$ by definition, and $\sympow{\mu}$ lifts to a volume form $\mu^{[n]}$ since the Hilbert--Chow morphism is crepant.  We may then form the modular vector field $\sympow{\tilde \zeta}$ of $\hilbps$ with respect to $\Hilb{\mu}$, which must agree with $\sympow{\tilde \zeta}$ over the smooth locus of $\sympow{\X}$, and hence it must be a lift of $\sympow{\tilde \zeta}$ (which is necessarily unique).   It follows that the modular flow on $\Hilb[n]{\X}$ is given by pulling back subschemes  $\Z\subset \X$ along the flow of $\tilde \zeta $ on $\X$.  The result now follows immediately, since $\tilde \zeta|_{\D} = \zeta$.
\end{proof}

\begin{remark}
Let us sketch an interpretation of this result in terms of derived symplectic geometry.  We first remark, following a discussion with P.~Safronov, that the modular vector field of a Poisson manifold can be thought of as the ``Hamiltonian vector field of the first Chern class of the leaf space'', in the following sense.  If $\Y$ is a one-shifted symplectic stack, then the symplectic form induces an isomorphism $\tc[\Y]\cong \ctc[\Y][1]$.  Under the induced isomorphism $\coH[1]{\ctc[\Y]}\cong \coH[0]{\ctc[\Y][1]}\cong\coH[0]{\tc[\Y]}$, the first Chern class $c_1(\Y)\in\coH[1]{\ctc[\Y]}$ corresponds to a canonical vector field $\zeta \in \coH[0]{\tc[\Y]}$ (its ``Hamiltonian vector field'').  If $p : \W \to \Y$ is a Lagrangian morphism, then $\W$ inherits a 0-shifted Poisson structure whose symplectic leaves are the fibres of $p$, and the vector field $\zeta$ agrees up to a constant factor with the projection of the modular vector field to the leaf space. In particular, the characteristic symplectic leaves in $\W$ correspond to fixed points of $\zeta$ in $\Y$. 

 In the case at hand, we have $\Y = \Perfo{\D}$, and $\tc[\Y]$ is the complex of traceless derived endomorphisms of the universal sheaf.  One can then compute the first Chern class of $\Y$ using the Grothendieck--Riemann--Roch theorem, and show that the corresponding vector field $\zeta$ generates the one-parameter group of automorphisms of $\Y$ given by pulling back complexes along the modular vector field on the curve $\D$ as above, recovering \autoref{prop:char-leaves}.
\end{remark}

\begin{corollary}\label{cor:singular-characteristic}
Let $\X$ be a Poisson surface with anticanonical divisor $\D$.  If $\Z \in \Hilb{\X}$ lies on a characteristic symplectic leaf, then $\Z \cap \D$ is contained in the singular locus of $\D$.
\end{corollary}

\begin{proof}
The modular vector field is nonzero at any smooth point of $\D$.  Hence if the support of $\Z \cap \D$ contains a smooth point, it will not be fixed by the flow of the modular vector field.
\end{proof}

\begin{lemma}\label{cor:codim2_char_leaves}
If $\X$ is connected, the codimension-two characteristic leaves are exactly the subvarieties of the form $\set{\Z \in \Hilb{\X}}{\Z\cap \D = \{p\}\textrm{ as schemes}}$ where $p$ is a singular point of $\D$.
\end{lemma}

\begin{proof}
If $\Z$ lies on a characteristic leaf of codimension two, then by \autoref{cor:singular-characteristic}, $\Z\cap \D$ must be contained in the singular locus of $\D$.  If  $\#(\Z\cap \D) > 1$, then the symplectic leaf through $\Z$ has codimension greater than two, so can assume that $\Z \cap \D =\{p\}$ for some singular point $p \in \D$.  Since $p \in \D$ is a singular point, it follows that $\Z$ itself is reduced at $p$, and hence $\Z =\{p\}\sqcup \Z'$ where $\Z' \subset \X\setminus \D$.  Hence the leaf is exactly an embedded copy of the symplectic variety $\Hilb[n-1]{(\X\setminus \D)}$. The transverse germ is isomorphic to the germ of $\X$ at $p$, and hence any such subvariety is a characteristic leaf.
\end{proof}

\begin{proposition}
If $\D$ has only nodal singularities, then each germ of $\Hilb{\X}$ has only finitely many characteristic symplectic leaves.
\end{proposition}

\begin{proof}
  The problem is local and invariant under rescaling the Poisson structure by a constant, so by the log Darboux theorem it is enough to treat the case in which $\X=\mathbb{C}^2$ and  $\ps = xy~\partial_x\wedge \partial_y$ (e.g. see \cite{Arnold1987}), in which case the modular vector field is given by $\zeta = y\cvf{y}-x\cvf{x}$ and its flow integrates to the $\CC^*$-action on $\X$ given by $t \cdot(x,y) = (tx,t^{-1}y)$.  It is straightforward to verify that this action preserves the equivalence class of the matrices $\mathscr{D}(\textbf{d})$, and  changes the equivalence classes of the matrices $M(\mathbf{d},k,u)$ by a suitable rescaling of $u$.  Hence according to \autoref{prop:orbit-data-node}, the orbit data for $\Hilb{\X}$ that are fixed by the modular flow are exactly those given by a single matrix from the series $\mathscr{D}(\textbf{d})$, of which only finitely many can occur in $\Hilb{\X}$.  Hence there are only finitely many characteristic leaves globally in $\Hilb{\X}$.  But the leaves are algebraic subvarieties (locally closed in the Zariski topology), so the analytic germ at every point can  have only finitely many connected components (bounded by the number of irreducible components in the analytic or formal germ of the leaf closure), as desired.
\end{proof}

For $n=1$, the characteristic leaves are given by the open leaf $\X\setminus \D$ and the singular points of $\D$, and hence are locally finite as long as $\D$ is reduced.  However if $\D$ has non-nodal singularities and $n$ is sufficiently large, one finds the existence of infinitely many characteristic symplectic leaves in $\Hilb{\X}$. Recall that a point $p \in \D$ is a \defn{double point} if the multiplicity of $\D$ at $p$ is exactly two, i.e.~the defining equation vanishes to order exactly two.), in which case it is an $A_k$ singularity for some $k \ge 1$, meaning that there exists coordinates such that $\D$ is given locally by the equation
$$
y^2-x^{k+1}=0.
$$
We have the following:
\begin{proposition}
Let $\X$ be a Poisson surface with anticanonical divisor $\D$.  Then the following statements hold:
\begin{enumerate}
\item $\Hilb[2]{\X}$ has finitely many characteristic symplectic leaves if and only if the only singularities of $\D$ are double points.
\item If $\D$ has a non-nodal singularity, then $\Hilb[n]{\X}$ has infinitely many characteristic leaves for every $n \ge n(\D)$ where
\[
n(\D) := \begin{cases}
2 & \textrm{if }\D \textrm{ has a triple point} \\
3 & \textrm{if }\D \textrm{ has a singularity of type }A_k\textrm{ for some  }k \ge 3 \\
6 & \textrm{if }\D \textrm{ has an }A_2\textrm{ singularity}.
\end{cases}
\]
\end{enumerate}
\end{proposition}

\begin{proof}
Suppose $\Z \in \Hilb[2]{\X}$.  If $\Z$ is reduced, then $\Z=\{p_1,p_2\}$ is a pair of distinct points and $\Hilb[2]{\X}$ is locally isomorphic to a neighbourhood of $(p_1,p_2)$ in the product $\X \times \X$, and hence there are only finitely many characteristic leaves in a neighbourhood of $\Z$.

If $\Z$ is not reduced, then it is the image of an embedding $\Spec{\CC[\epsilon]/\epsilon^2}\to \X$ and is therefore classified by a pair $(p,\bL)$ where $p \in \X$ and $\bL < \tb[p]{\X}$ is a line.  By \autoref{cor:singular-characteristic} we need only consider the case in which $p$ is a singular point of $\D$, in which case $\Z$ is contained entirely in $\D$, so that $\ps$ vanishes at $\Z$.  Let us choose local coordinates $x,y$ on $\X$ centred at $p$ so that the bivector on $\X$ has the form $\ps=f(x,y)~\cvf{x}\wedge\cvf{y}$.  The modular vector field is then given in terms of the partial derivatives $f_x,f_y$ by the formula
\[
\zeta = f_x \cvf{y} - f_y \cvf{x}
\]
and hence the linearization of $\zeta$ at $p$ is given by the matrix
\[
\begin{pmatrix}
-f_{xy}(0) & -f_{yy}(0) \\
f_{xx}(0) & f_{xy}(0)
\end{pmatrix}
\]
If $f$ vanishes to order three or more, then this matrix is zero.  Hence every line $\bL < \tb[p]{\X}$ is preserved by the flow of $\zeta$, and we obtain infinitely many characteristic leaves in $\Hilb[2]{\X}$, and hence also in $\Hilb[n]{\X}$ for $n > 2$.  On the other hand, if $f$ vanishes to order two, this matrix is nonzero.  Since it is traceless, it preserves at most two lines in $\tb[p]{\X}$ and so the set of characteristic leaves with support at $p$ is finite.

Now suppose that $\D$ has an $A_k$-singularity  at $p \in \X$.  Then we may choose coordinates $(x,y)$ centred at $p$ such that $f(x,y) = h(x,y) (y^2-x^{k+1})$ and $h(0,0)\not=0$. Then the modular vector field of $\pi$ is given by
$$
\zeta = h(x,y) \Big(- 2y \partial_x -(k+1)x^k \partial_y \Big) + (y^2-x^{k+1}) \Big(h_x(x,y)\partial_y-h_y(x,y)\partial_x\Big).
$$
Suppose first that $k\ge 3$, and consider the ideal $\cI_a:=(y+ax^2, x^3)$ where, $a\in\mathbb{C}$. Then each $\cI_a$ defines a point in $\Hilb[3]{\X}$ and contains the ideal $\cI_\D = (y^2-x^{k+1})$ defining $\D$.  Therefore, multiplication by $f$ annihilates ${\rm Hom}(\mathcal{I}_a,\mathcal{O}_\X/\mathcal{I}_a)$, and so $\Hilb[3]{\ps}$ vanishes at $\mathcal{I}_a$. Moreover, the modular vector field $\zeta$ sends each generator of $\mathcal{I}_a$ to another element of $\mathcal{I}_a$. Therefore, the vector field $\Hilb[3]{\zeta}$ on $\Hilb[3]{\X}$ vanishes at $\mathcal{I}_a$, so that the singleton $\{\cI_a\}$ is a characteristic leaf for all $a \in \CC$.

Finally, assume $k=2$. Then we can assume $h(x,y)=1$ in the expression for $\ps$ above by a theorem of Arnol'd~\cite{Arnold1987}.  For $a \in \CC$, consider the  ideal $\mathcal{J}_a = (y^2 + ax^3,x^2y, xy^2)$, which defines an element of $\Hilb[6]{\X}$. We claim that multiplication by $f=y^2-x^3$ annihilates each element $\varphi \in {\rm Hom}(\mathcal{J}_a, \mathcal{O}_\X/\mathcal{J}_a)$, for each $a\in\mathbb{C}$. Denoting the generators $g_1=y^2 + ax^3$, $g_2=x^2y$, $g_3=xy^2$, we have
\begin{align*}
y\varphi(g_2) - x \varphi(g_3) &= 0 &
 -xy\varphi(g_1) + ax^2\varphi(g_2) + y \varphi(g_3) &= 0.
\end{align*}
It is easy to deduce from these equations that the image of any such $\varphi$ lies inside $(x,y) \mathcal{O}_\X/\mathcal{J}_a$, which is clearly annihilated by multiplication by $f$. Therefore, the Poisson tensor $\Hilb[6]{\ps}$ vanishes at $\mathcal{J}_a$. Moreover, the vector field $\zeta=-2y\partial_x -3x^2 \partial_y$ sends each generator of $\mathcal{J}_a$ to another element of $\mathcal{J}_a$. Therefore, $\Hilb[6]{\zeta}$ vanishes at $\mathcal{J}_a$, so that the singleton $\{\mathcal{J}_a\}$ is a characteristic symplectic leaf for each $a \in \CC$.
\end{proof}

\section{Local models and holonomicity}
\label{sec:local}

\subsection{Coordinates on the Hilbert scheme}

There are two well-known methods for constructing coordinate charts on the Hilbert schemes, indexed by Young diagrams: the Haiman coordinates~\cite{Hai98}, and the \ES coordinates~\cite{Ellingsrud1988}.  In this section we use these coordinate systems to describe the local behaviour of Bottacin's Poisson structures.  Of particular relevance to us are the charts associated to the ``triangular'' Young diagrams
\[
\ytableausetup{centertableaux,smalltableaux}
\begin{ytableau}
\\
\end{ytableau}
\qquad
\begin{ytableau}
\\
&\\
\end{ytableau}
\qquad
\begin{ytableau}
\\
&\\
&&\\
\end{ytableau}
\qquad
\begin{ytableau}
\\
&\\
&&\\
&&&\\
\end{ytableau}
\qquad \cdots
\]
which are defined as follows.

Choose a point $p \in \X$ and an integer $k > 0$, and let $p(k)$ be the $k$-th order neighbourhood of $p$, i.e.~the vanishing locus of the ideal $\m_p^{k+1} < \cO{\X}$ where $\m_p$ is the maximal ideal corresponding to $p$.  In local coordinates $(x,y)$ on an open set $\U$ centred at $p$, the scheme $p(k)$ is expressed as the vanishing locus of the monomials $x^jy^l$ where $j+l > k$, so that $\cO{p(k),p}$ has a basis given by the monomials $x^jy^l$ with $j+l\le k$.  The latter are in bijection with the boxes of the triangular Young  diagram
\[
\TYD := (k > k-1 > k-2 > \cdots)
\]
so that
\[
p(k) \in \Hilb[n]{\X} 
\]
where
\[
n := |\,\TYD| = \tfrac{1}{2}k(k+1)
\]
is the size of $\TYD$.  The point $p(k) \in \Hilb{\X}$ then has an open neighbourhood given by 
\[
\Utri := \set{\Z \in \Hilb{\X}}{\textrm{the monomials }x^jy^l, j+l<k\textrm{ form a basis for }\coH[0]{\cO{\Z\cap \U}}}
\]
and the two coordinate systems on these charts are defined as follows.
\begin{itemize}
\item \textbf{\ES coordinates:} If $E \in \CC^{(k+1)\times k}$ is a $(k+1)\times k$ matrix, let
\begin{align}
S_E(x,y) :=  E + \begin{pmatrix}
-x & 0 & 0 & \cdots & 0 \\
y & -x & 0 & \cdots & 0 \\
0 & y & -x & \cdots & 0 \\
0 & 0 & y & \ddots & \vdots \\
\vdots & \vdots&\vdots & \ddots & -x \\
0 & 0 & 0 & \cdots & y
\end{pmatrix}  = E-x \begin{pmatrix}
I_k \\
0
\end{pmatrix} + y \begin{pmatrix}
0 \\ I_k
\end{pmatrix} \label{eq:ES-syzygy}
\end{align}
and define a subscheme $\Z(E) \subset \X$ by the formula
\[
\Z(E) := \textrm{vanishing locus of the maximal minors of }S_E(x,y) 
\]
Then the map $E \mapsto \Z(E)$ identifies a neighbourhood of the origin in $\CC^{(k+1)\times k}$ with $\Utri$.  Note that $S_E(x,y)$ is a Hilbert--Burch syzygy matrix for $\Z(E)$.  The functions
\[
E^j_i : \Utri \to \CC, ~~~~ 0\le i \le k-1, ~0 \le j \le k
\]
sending $\Z(E)$ to the $(j,i)$-entry of the corresponding matrix $E$ are the \ES coordinates on $\Utri$. Note that here we start numbering of rows and columns from $0$.
\item \textbf{Haiman coordinates:} For every $\Z \in \Utri$, the ideal $\cI_\Z$ is generated by functions of the form
\begin{align}
f_{\Z,j} = x^j y^{k-j} - \sum_{a+b\le k-1} c_{ab}^j(\Z) x^a y^b \qquad 0 \le j \le k. \label{eq:Haiman-relation}
\end{align}
where $c_{ab}^j$ are certain regular functions on $\Utri$, called Haiman's functions.  Haiman's coordinates are then given by the collection
\[
C_i^j := c_{i,k-1-i}^j
\]
for $0\le i \le k-1$ and  $0\le j \le k$.
\end{itemize}
These coordinate systems are related by the linear transformations
\begin{align*}
E_i^j &= C_{j}^{i+1} - C_{j-1}^{i} 
\\
C_i^j &= \begin{cases}
E^i_{j-1} + E_{j-2}^{i-1} + \cdots + E_{j-i-1}^0 & j\ge i+1 \\
 - E_{j}^{i+1} -  E_{j+1}^{i+2} - \cdots - E_{j+k-i-1}^{k} & j\le i,
 \end{cases}
\end{align*}
where we use the convention that $C_{-1}^j=C_{k}^j = 0$ for any $j$.   One way to explicitly verify the second transformation is by expanding the maximal minors of $S_E(x,y)$,
considering only the linear term in $E$ (equivalently, total degree $k-1$ in $x,y$). 

\begin{remark}\label{rk:coordinates_come_in_halves}
Note that each set of coordinates $E_i^j$ and $C_i^j$ splits into two halves: $j\le i$ and $j\ge i+1$. In the formula above, the $E$-coordinates in one half are expressed in terms of the $C$-coordinates in the other half, and vice versa.
\end{remark}

\begin{remark}
In the particular case where $\X = \CC^2$ and $(x,y)$ are the standard coordinates, these charts are defined on the whole space of $(k+1)\times k$-matrices, giving open embeddings $\CC^{(k+1)\times k} \hookrightarrow \Hilb{(\CC^2)}$ for all $k > 0$.
\end{remark}

An important feature of these coordinate charts is that their construction is $\CCx$-equivariant, in the following sense.
 Define an action of $\CCx$ on the germ of $\X$ at $p$ by dilation of the coordinates we chose:
\[
u \cdot(x,y) = (ux,uy) \qquad u \in \CCx.
\]
By functoriality, we have an induced action on the germ of $\Hilb{\X}$ at $p(k)$.
\begin{lemma}\label{lem:homogeneity-of-coords}
The \ES and Haiman coordinates have weight one with respect to the induced $\CCx$-action, i.e.~
\[
u \cdot E_i^j = uE_i^j \qquad u \cdot C_i^j = uC_i^j.
\]
for all indices $i,j$ and all $u \in \CCx$.
\end{lemma}
\begin{proof}
Since the two coordinate systems are related by a linear transformation, it suffices to prove the statement for one of them; we do so for the \ES coordinates.  For this, note that 
\[
S_E(ux,uy) = E - u x \begin{pmatrix}
I_k \\ 0
\end{pmatrix} + uy \begin{pmatrix}
0 \\ I_k
\end{pmatrix} = u\rbrac{u^{-1}E -  x \begin{pmatrix}
I_k \\ 0
\end{pmatrix} + y \begin{pmatrix}
0 \\ I_k
\end{pmatrix}} = uS_{u^{-1}E}(x,y)
\]
Hence the ideals generated by the minors of the matrices $S_E(ux,uy)$ and $S_{u^{-1}E}(x,y)$ are the same, which implies that the \ES coordinates scale linearly, as desired.
\end{proof}

\subsection{Expression for the Poisson structure}
\label{sec:triangular-chart-Poisson}
Our interest in the triangular charts stems, in part, from the fact that they exhibit all possible local behaviours of Bottacin's Poisson structures. Indeed recall from \cite{Matviichuk2020} that two germs $\W_1$ and $\W_2$ of Poisson manifolds are \defn{stably equivalent} if there exist symplectic germs $\bS_1$ and $\bS_2$ such that $\W_1 \times \bS_1 \cong \W_2 \times \bS_2$.  Then we have the following:
\begin{lemma}\label{lm:Hilb_equiv_to_triang_chart}
Let $(\X,\ps)$ be a Poisson surface and suppose $\Z \in \Hilb[m]{\X}$ for some $m \ge 0$.  Then the germ of $(\Hilb{\X},\hilbps)$ at $\Z$ is stably equivalent to a product of germs of points in the triangular charts of Hilbert schemes $\Hilb[k(k+1)/2]{\X}$, $k \ge 0$
\end{lemma}

\begin{proof}
Any element $\Z \in \Hilb[m]{\X}$ has an analytic neighbourhood consisting of products of neighbourhoods of its connected components, so we may assume without loss of generality that $\Z$ is supported at a single point $p\in \D$.  If suffices to show that there exists an integer $k >0$ and a subscheme $\Z' \subset \X\setminus \D$ such that $\Z\sqcup \Z'$ lies in the triangular chart $\Utri\subset\Hilb[k(k+1)/2]{\X}$, since in this case the germ at $\Z'$ is symplectic and hence the germs at $\Z$ and at $\Z\sqcup \Z'$ are stably equivalent.

To see this choose local coordinates on an open set $\U$ centred at $p$ and choose $k$ large enough that $\coH[0]{\cO{\Z}}$ is spanned by the images of the monomials $A = \set{x^jy^l}{j+l<k} \subset\coH[0]{\cO{\U}}$.  Now choose a  collection of distinct points $p_1,p_2,\ldots \in \U\setminus(\U\cap \D)$ such that the evaluation maps $ev_{p_j} : \coH[0]{\cO{\U}}\to \CC$ form a basis for $\mathrm{span}(A)^\vee$.  After reordering, we can assume by linear independence that $ev_{p_1},\ldots,ev_{p_{k-m}}$ restrict to a basis of $\rbrac{\ker\rbrac{ \mathrm{span}(A) \to \coH[0]{\cO{\Z}}}}^\vee=\coker\rbrac{\coH[0]{\cO{\Z}}^\vee \to\mathrm{span}(A)^\vee}$, and then $\Z' = \{p_1,\ldots,p_{k-m}\}$ is the desired subscheme. 
\end{proof}

An explicit expression for the Poisson bracket can be obtained via the following algorithm:
\begin{enumerate}
\item First obtain a formula in the case of the Poisson structure $\ps_0 := \cvf{x}\wedge\cvf{y}$; we do so in \autoref{sec:Darboux} below.
\item Define the \defn{recursion operators}  $J_x, J_y$ to be the endomorphisms of the tangent bundle given by the action of $x,y$ on the tangent spaces under the identification
\[
\tb[\Z]{\Utri} \cong \Hom[\U]{\cI_\Z,\cO{\Z}}
\]
Note that these operators have weight one, which implies they are linear in \ES or Haiman coordinates.  Explicitly, they are given as follows (we postpone their derivation to  \autoref{subsec:recursion}):

\begin{align}
\begin{split}
J_x\cdot \cvf{C_i^j} &= \sum_{b=0}^k E_i^b \cvf{E_{j-1}^b} - \sum_{a=0}^{k-1} E_a^j \cvf{E_a^{i+1}},
\\ 
J_y \cdot \cvf{{C_i^j}} &= \sum_{b=0}^k E_i^b\cvf{E_j^b} - \sum_{a=0}^{k-1} E_a^j \cvf{E_a^i}.
\end{split}\label{eq:recusions-ops}
\end{align}

\item Now given an arbitrary Poisson structure
\[
\ps = f(x,y)\cvf{x}\wedge\cvf{y} = f(x,y)\ps_0
\]
on $\U$, it follows from Bottacin's formula that $\hilbps$ is the bivector corresponding to the
linear map $(\hilbps)^\sharp : \ctb{\Utri} \to \tb{\Utri}$
given by the formula
\[
(\hilbps)^\sharp = f(J_x,J_y)(\hilbps_0)^\sharp.
\]
Note that the pair $(\hilbps,\hilbps_0) $ is a nondegenerate bi-Hamiltonian structure and  $f(J_x,J_y)$ is its associated recursion operator in the sense of integrable systems; this is the source of our name for $J_x$ and $J_y$. 
\end{enumerate}

In the following sections we describe the local structure in the cases in which $\D$ has at worst nodal singularities; by the Darboux theorem and its generalizations, this corresponds to the case in which $\ps$ is homogeneous of nonpositive weight in a suitable chart $(x,y)$.

\subsection{Darboux coordinates}
\label{sec:Darboux}
Suppose that $p \in \X \setminus \D$.  Then $\ps$ is nondegenerate at $p$, so by the Darboux theorem there exists coordinates $(x,y)$ centred at $p$ such that 
\[
\ps = \cvf{x}\wedge\cvf{y}.
\]
In particular $\ps$ has weight $-2$, i.e.~is constant, and hence the corresponding Poisson structure $\hilbps$ is induced by a constant symplectic structure $\Hilb{\omega}$ in the chart $\Utri$.  Natural Darboux coordinates for $\Hilb{\omega}$ are then given by the following.
\begin{proposition}\label{prop:Darboux}
The \ES and Haiman coordinates on $\Utri$ are symplectically dual, in the sense that
\[
\{E_{i}^j,C_{i'}^{j'}\} = \begin{cases}
1 & i=i' \textrm{ and } j=j' \\
0 &\textrm{otherwise}
\end{cases}
\]
Hence $E_i^j, C_i^j$, $j\le i$, form Darboux coordinates, i.e. ~
\[
\omega^{[n]} = \sum_{0\le j\le i \le k-1} dC_i^j \wedge dE_i^j.
\]
\end{proposition}

\begin{proof}
$\Utri$ has an open dense set consisting of reduced schemes $\Z$ lying in $\U$.  Near such an element we have Darboux coordinates $x_1,y_1,\ldots,x_n,y_n$ indicating the $x$ and $y$ coordinates of the support the subscheme $\Z'$ near $\Z$.  Thus
$$
\omega^{[n]} = \sum_{i=1}^n dy_i \wedge dx_i = {\rm Tr} ( dM_y \wedge dM_x),
$$
where $M_x$ and $M_y$ are the matrices of functions representing the operators of multiplication by $x$ and $y$ on $\coH[0]{\cO{\Z}}$ in the basis of indicator functions of the points of $\Z$.  By \autoref{lm:tr(dAdB)} below, the right-hand side may in fact by computed using the same formula with respect to \emph{any} basis for $\coH[0]{\cO{\Z}}$.  In particular, in $\Utri$ we may choose the basis of monomials $e_{ij} := x^i y^j$, $i+j\le k-1$, and compute the actions $x$ and $y$ using Haiman's defining relations \eqref{eq:Haiman-relation} for $\cO{\Z}$:
\begin{align*}
x \cdot e_{ij} &= \begin{cases}
e_{i+1,j} & i+j\le k-2,  \\
 \sum_{a+b\le k-1} c_{ab}^k e_{ab} & i+j=k-1
 \end{cases}\\
y \cdot  e_{ij}&= \begin{cases} e_{i,j+1} &\qquad i+j\le k-2 \\
 \sum_{a+b\le k-1} c_{ab}^k e_{ab} &\qquad i+j=k-1,
\end{cases}
\end{align*}
from which we deduce that
$$
\omega^{[n]} = {\rm Tr} ( dM_y \wedge dM_x) = \sum_{i+j,a+b\le k-1}  d\langle y \cdot e_{ij}, e_{ab}^* \rangle \wedge d \langle  x\cdot e_{ab}, e_{ij}^*\rangle = \sum_{i,a=0}^{k-1} dC_a^{i} \wedge d C_i^{a+1}.
$$
Therefore
\begin{equation}\label{eq:darboux_Hilb_1}
 \Hilb{\omega}(\cvf{C_i^j},-) =  dC_{j}^{i+1} - dC_{j-1}^{i} = dE_i^j, 
\end{equation}
so that the Hamiltonian vector field of $E_i^j$ is $C_i^j$, as claimed.

To prove the claim about the Darboux coordinates, note that due to \autoref{rk:coordinates_come_in_halves} we have $\{E_i^j, E_{i'}^{j'}\} = \{C_i^j, C_{i'}^{j'}\} = 0$ for all $j\le i$, $j'\le i'$.
\end{proof}

\begin{lemma}\label{lm:tr(dAdB)}
Let $A$, $B$ be a pair of commuting square matrices with functional entries. Then the two form ${\rm Tr}(dA \wedge dB)$ is invariant under simultaneous conjugation of $A$ and $B$ via an invertible matrix with functional entries.
\end{lemma}

\begin{proof} 
We learned this from the preprint version of \cite[proof of Proposition C.1]{Thom}.  Since the published version does not contain the proof, we include it here for completeness.  Let $\widetilde{A}=gAg^{-1}$, $\widetilde{B}=gBg^{-1}$ for some invertible matrix of functions $g$. Then we have
$$
{\rm Tr} ( d\widetilde{A} \wedge d\widetilde{B}) = {\rm Tr} (dA \wedge dB)  - {\rm Tr}~ d \Big([A,B] ~g^{-1}  dg\Big) . 
$$
which gives the result.
\end{proof}

\subsection{Recursion operators}\label{subsec:recursion}

In this subsection, we derive the formulae for the recursion operators~\eqref{eq:recusions-ops} in the triangular chart of the Hilbert scheme of $\X=\mathbb{C}^2$.  Note that since $x$ and $y$ have weight one with respect to the natural $\CCx$-action, the operators $J_x$ and $J_y$ also have weight one, and in particular they send the basis vector fields $\cvf{C_i^j}$ in Haiman coordinates to linear vector fields.  
By applying the automorphism that interchanges $x$ and $y$ (which has the form $E^j_i \mapsto -E^{k-j}_{k-i-1}$ in \ES coordinates{}) the calculation of $J_x$ reduces to that of $J_y$, so we will just explain how to do the latter.

The tangent vector $\cvf{C_i^j}$, $0\le i \le k-1$, $0\le j \le k$, corresponds to an element $\varphi_i^j \in \Hom{\cI,\cO{\Z}}$ of the form
$$
\varphi_i^j(f_r) = \sum_{a+b\le k-1} \gamma_{abi}^{rj} ~x^a y^b ~~~{\rm mod~} \mathcal{I},
$$
where $f_r=f_{\Z,r}$, $r=0,1,...,k$, are the generators \eqref{eq:Haiman-relation} of $\mathcal{I}$. Note that each $f_r$ is a homogeneous polynomial of degree $k$ in variables $x$, $y$ and the Haiman coordinates $C_i^j$. Therefore, each $\gamma_{abi}^{rj}$ is a homogeneous polynomial of degree $k-1-a-b$ in Haiman coordinates. By definition of the Haiman coordinates,
we have

\begin{equation}\label{eq:tangent_fiber_via_Hom(I,O_Z)}
\varphi_i^j(f_r) = \delta_{jr} x^i y^{k-1-i} + \sum_{a=0}^{k-2} \theta_{ai}^{rj} ~x^a y^{k-2-a} \mod V_3
\end{equation}
where $\delta_{jr}$ is the Kronecker delta symbol, $\theta_{ai}^{rj}$ are linear functions in Haiman coordinates, and for $l \ge 0$ we denote by
\[
V_l := \cI \oplus \bigoplus\limits_{c+d\le k-l}\mathbb{C}\cdot x^cy^d \subset \CC[x,y]
\]
the linear subspace generated by $\cI$ and the monomials of degree at most $k-l$.

By using the equalities $\varphi_i^j\big(x f_r - y f_{r+1}\big)  = x\varphi_i^j(f_r) - y \varphi_i^j(f_{r+1}) $, $r=0,1,...,k-1$, and looking the terms modulo $V_2$, we obtain the following linear equations on $\gamma$: 
\begin{align*}
\theta_{ai}^{r+1,j} - \theta_{a-1,i}^{rj} &= \delta_{rj} C_{a}^{i+1}- \delta_{r+1,j}  C_{a}^{i} - \delta_{ai}\big(C_{j}^{r+1}-C_{j-1}^{r}\big), &  0\le i,r \le k-1, 0 \le j,a \le k.
\end{align*}
Here and below we adopt the convention that if an index of $C$ or $\theta$ falls outside the allowed range, then we declare the value of such $C$ or $\theta$ to be zero.

Finally, to compute  $J_y \cdot \cvf{C_i^j} $, we need to calculate $y\varphi_i^j$. By \eqref{eq:tangent_fiber_via_Hom(I,O_Z)}, one has
\begin{align*}
y\varphi_i^j(f_b) &= \delta_{jb} x^i y^{k-i} + \sum_{a=0}^{k-2} \theta_{ai}^{bj} ~x^a y^{k-1-a} \\
&= \delta_{bj} \sum_{a=0}^{k-1} C_{a}^{i} x^a y^{k-1-a}  + \sum_{a=0}^{k-2} \theta_{ai}^{bj} x^a y^{k-1-a} \\
&=\sum_{a=0}^{k-1} \Big(\delta_{bj}C_{a}^{i}  + \theta_{ai}^{bj}\Big) x^a y^{k-1-a} \mod V_2
\end{align*}
This implies that
$$
J_y \cdot \cvf{C_i^j}= \sum_{b=0}^k \sum_{a=0}^{k-1} \Big(\delta_{bj}C_{a}^{i}  + \theta_{ai}^{bj}\Big) \dfrac{\partial}{\partial C_a^b},
$$
and therefore
\begin{align*}
\abrac{ J_y \cdot \cvf{C_i^j}, dE_a^b } &= \abrac{\sum_{\beta=0}^k \sum_{\alpha=0}^{k-1} \Big(\delta_{\beta j}C_{\alpha}^{i}  + \theta_{\alpha i}^{\beta j}\Big) \dfrac{\partial}{\partial C_\alpha^\beta} , dC_{b}^{a+1} - dC_{b-1}^a }
  \\
&=\delta_{a+1,j} C_b^i - \delta_{aj} C_{b-1}^i + \theta_{b,i}^{a+1,j} - \theta_{b-1,i}^{a,j}\\
& = \delta_{a+1,j} C_b^i - \delta_{aj} C_{b-1}^i +
\delta_{aj} C_{b}^{i+1}- \delta_{a+1,j}  C_{b}^{i} - \delta_{bi}\big(C_{j}^{a+1}-C_{j-1}^{a}\big) \\
&= \delta_{aj} (C_b^{i+1} - C_{b-1}^i) - \delta_{bi}\big(C_{j}^{a+1}-C_{j-1}^{a}\big) \\
&= \delta_{aj} E_i^b - \delta_{bi}E_a^j,
\end{align*}
as desired.

\subsection{Smooth curves and the Lie algebra $\aff{k}$}

Now suppose that $p \in \D$ is a smooth point.  Then by the log Darboux theorem there exist coordinates $(x,y)$ centred at $p$ such that $\D$ is given by $y=0$ and
\[
\ps = y \cvf{x}\wedge \cvf{y}
\]
In particular $\ps$ is linear in these coordinates, and hence $\hilbps$ is linear in the corresponding Haiman/\ES coordinates, so that it corresponds dually to a Lie algebra structure on the space of $(k+1)\times k$ matrices.

We claim that this is precisely the Lie algebra $\aff{k} \cong \gln[k]{\CC}\ltimes \CC^k$ of the group  $\Aff{k} \cong \GL[k]{\CC} \ltimes \CC^k$ of affine transformations of $\mathbb{C}^k$.  To see this, we use the natural embedding
\[
\psi : \Aff{k} \hookrightarrow \GL[k+1]{\CC}
\]
given by
\[
\psi(g,v) = \begin{pmatrix}
g & v \\
0 & 1
\end{pmatrix}
\]
where $g \in \GL[k]{\CC}$ and $v \in \CC^k$, which identifies $\aff{k}$ with $\CC^{(k+1)\times k}$.  Identifying the duals via the trace pairing, the coadjoint action of $\Aff{k}$ on the vector space $\aff{k}^\vee \cong \CC^{(k+1)\times k}$ is given by the formula
\[
(g,v) \cdot E := \psi(g,v) E g^{-1}
\]
Now if $E \in \CC^{(k+1)\times k}$, the syzygy matrix of the scheme $\Z(E)$ is the matrix $S_E(x,y)$ from \eqref{eq:ES-syzygy}, and hence $i^*\cI_{\Z(E)}$ is presented as the cokernel of the matrix
\[
S_E(x,y)|_\D = S_E(x,0) = E - x\begin{pmatrix}
I_k \\ 0
\end{pmatrix}
\]
and direct computation shows that
\[
\psi(g,v) S_E(x,0)  g^{-1} = S_{(g,v)\cdot E}(x,0),
\]
so that $(g,v)$ induces a $\D$-equivalence from $\Z(E)$ to $\Z((g,v)\cdot E)$.  This defines a map of groupoids between $\Aff{k}\ltimes \aff{k}$ and $\GHilb{\X}{\D}$ over the open sets corresponding to $\Utri$, giving a local model for the symplectic groupoid $\GHilb{\X}{\D}$:
\begin{theorem}\label{thm:triang_chart_Hilb_y_affk}
The map $\E \to \Z(E)$ gives a Poisson isomorphism from an open set in $\aff{k}^\vee$ to an open set in $\Utri$, and the corresponding map of groupoids is a symplectomorphism onto its image.
\end{theorem}

\begin{proof}
We first prove that the map $E \mapsto \Z(E)$ is Poisson.  For this we compute the Poisson bivector on $\Utri$ following the algorithm of \autoref{sec:triangular-chart-Poisson}.  From the formulae in \autoref{prop:Darboux} and \eqref{eq:recusions-ops} we have  $\{E_i^j,E_a^b\} = \abrac{ J_y \cdot \cvf{C_i^j}, dE_a^b } = \delta_{aj} E_i^b - \delta_{bi}E_a^j$, which is exactly the formula for the bracket on $\aff{k}^\vee$ in the given basis, as desired.  Hence the map $E \mapsto \Z(E)$ is Poisson.

This implies immediately that the map of groupoids is compatible with symplectic structures, and hence it is \'etale onto its image.  It remains to check that it is injective.  For this, we simply note that any element $(g,v)\in \Aff{k}$ can be recovered from the corresponding $\D$-equivalence by evaluating the corresponding bundle endomorphism of $\cO{}^{k+1}$ at $p$ (which is well-defined modulo $\D$-equivalence, i.e., depends only on the homotopy class of the map of cochain complexes).
\end{proof}

\begin{corollary}\label{cor:nilcone}
For $\X = \U = \CC^2$ and $\ps = y\cvf{x}\wedge\cvf{y}$, we have $\Utri \cong\aff{k}^\vee$ as Poisson manifolds. Under this identification:
\begin{enumerate}
\item The symplectic leaves in $\Utri$ are identified with the coadjoint orbits of $\aff{k}$.
\item\label{cor:nilcone:Ak} The linear subspace $\gln[k]{\CC}^\vee \subset \aff{k}^\vee$ corresponds to the set $\mathcal{A}_k$ of elements $\Z \in \Utri$ such that $\#(\Z\cap \D) = k$.  
\item The nilpotent cone in $\gln[k]{\CC}^\vee$ is identified with the set of elements $\Z \in \Utri$ such that $\cI_{\Z\cap \D} = (x^k,y)$. 
\end{enumerate}
\end{corollary}

\begin{proof}
It remains to establish points 2 and 3 above.  For this, note that the subspace $\gln[k]{\CC}^\vee$ is cut out by the $k$ equations $E_i^k=0$, $i=0,1,...,k-1$.  If $E_i^k\not=0$ for some $0\le i \le k-1$, then the $(k-i)$-th minor of $S_E(x,0)$ is a polynomial whose highest-degree term is $\pm E_i^k x^{k-1}$.  Since this function vanishes on $\Z \cap \D$, we deduce that $\# \Z \cap \D < k$.  Conversely, if $E_i^k=0$ for all $i$, then all maximal minors of $S_E(x,0)$ are zero, except the last one, which is equal to the characteristic polynomial of the square matrix $\tilde E := (E_i^j)_{i,j=0}^{k-1}$, which has degree exactly $k$.  Hence $\#(\Z\cap \D) = k$, which establishes statement 2.  Statement 3 follows immediately since $\tilde E$ lies in the nilpotent cone if and only if its characteristic polynomial is $x^k$.
\end{proof}

\begin{remark}\label{r:groth-spr}
  More generally, if $(\X,\ps)$ is a Poisson surface with smooth zero divisor $\D$, then for any $k\le n$ one can define a locally closed Poisson submanifold $\W \subset \X^{[n]}$ as the locus of elements $\Z \in \Hilb{\X}$ such that $\#(\Z\cap \D) = k$, and one has a natural projection $\W \to \sympow[k]{\D}$ whose fibres are Poisson subvarieties of $\Hilb{\X}$.  Away from the big diagonal in $\sympow[k]{\D}$, the fibres are smooth and symplectic, but in general they can be singular.   \autoref{cor:nilcone} implies that the  singularities 
  are products of slices of nilpotent orbits in $\sln{m}(\CC)$, $m\le k$.  These singularities are all symplectic.  Near a most singular fibre $(\Z \cap \D) = k \cdot p$ with $n \ge \tfrac{1}{2}k(k+1)$, the family restricts to the universal Poisson deformation of the nilpotent cone  (times a parameter giving the centre of mass of $\Z$, which makes sense in local coordinates).  We expect it to have a unique simultaneous resolution $\widetilde{\W} \to \W$, giving a global analogue of the Grothendieck--Springer resolution, which restricted to the fibre over $k \cdot p$ is a global analogue of the Springer resolution. It would be interesting to further study these objects.
  For example, are all Poisson deformations of $\W, \widetilde{\W}$ given by deformations of the Hilbert scheme? And are all Poisson deformations of the fibre over $k \cdot p \in D$ given by deforming the Hilbert scheme and moving the point in $\sympow[k]{\D}$ (at least infinitesimally)?
\end{remark}

Using this result, we obtain a complete characterization of the closure relation between symplectic leaves as a condition on the corresponding Young diagrams.  Note that the necessity of this condition was proven by Rains in \cite[Section 11.2]{Rains2016}, but our proof of sufficiency uses our local normal form in an essential way.

\begin{proposition}\label{cor:Hilb_y_sympl_leaves_order}
Let $(\X,\ps)$ be a Poisson surface with a smooth zero divisor $\D$, and let $\mu,\tmu$ be global orbit data as in \autoref{lem:smooth-orbit-data}.  Then the closure of  the symplectic leaf $\Hilb[n]{\X}_{\mu}$ contains $\Hilb[n]{\X}_{\widetilde{\mu}}$ if and only if for every $p \in \D$, we have the inclusion $\hc{{\mu}(p)}\subseteq \hc{\widetilde{\mu}(p)}$, or equivalently $\mu(p)\le\widetilde{\mu}(p)$ in the dominance order on Young diagrams.
\end{proposition}

\begin{proof}
  The second equivalence is a simple consequence of the definitions. Let us prove the first equivalence.

Let us start by remarking for any ideal $\Z \in \X^{[m]}$, $m < n$,  the set 
$$\big\{
\Z' \in \Hilb[n]{\X} : \Z'\supseteq \Z
\big\}
$$
is a closed subvariety. By \autoref{thm:sympl_leaves_Hilb_y}, every element $\Z\in \Hilb[n]{\X}_{\mu}$ satisfies 
$\Z \supset \bigsqcup_{p\in \D} \Z^{\mu(p)}_p$. Therefore, if $\Hilb[n]{\X}_{\widetilde{\mu}}$ lies in the closure of $\Hilb[n]{\X}_\mu$, then we have 
$\Z^{\widetilde{\mu}(p)}_p\supseteq \Z^{\mu(p)}_p$ for each $p\in \D$.

For the converse, suppose that that $\mu\subseteq \widetilde{\mu}$.  We wish to show $\Hilb[n]{\X}_{\tmu} \subset \overline{\Hilb[n]{\X}_\mu}$. It suffices to prove this result in the complex topology, since it is finer.

It is enough to consider the case when both $\mu$ and $\widetilde{\mu}$ consist of one non-trivial Young diagram each, concentrated at the same point $p \in \D$.  By abuse of notation, we denote these diagrams simply by  simply by $\mu$ and $\widetilde{\mu}$.  We can also assume without loss of generality that $\X=\mathbb{C}^2$ equipped with the Poisson bivector, $\ps = y~\partial_x \wedge \partial_y$, and $p=(0,0)$.  Furthermore, note that by \autoref{thm:sympl_leaves_Hilb_y} it suffices to prove the closure relation for the open dense sets consisting of subschemes of the form $\Z^{\mu(p)}_p \sqcup \Z'$ where $\Z' \subset \X \setminus \D$ is reduced.  Such an element has an analytic neighbourhood of the form $\Hilb[k]{\U} \times \Hilb[n-k]{\V}$ where $\U$ is an analytic neighbourhood of $p$ and $\V$ is an analytic neighbourhood of $\Z'$ such that $\overline{\V} \subset \X\setminus \D$ and $\overline{\U} \cap \overline{\V} = \varnothing$.  Hence by adding or removing points from $\Z'$, we see that it is enough to prove the statement for any single value of $n \ge |\hc{\tmu}|$.

Let $\lambda=\hc{\mu}$ and $\widetilde{\lambda} = \hc{\widetilde{\mu}}$.  We consider two special cases for the pair $(\lambda,\tlambda)$ below, and then explain how the general case follows from these. 

\emph{Case 1:} $\lambda_1 < \widetilde{\lambda}_1$, but $\lambda_i= \widetilde{\lambda}_i$, for $i> 1$.  In this case, by induction on the difference $\widetilde{\lambda}_1 - \lambda_1$, we may assume without loss of generality that $\lambda_1 + 1 = \widetilde{\lambda}_1$. In other words, the Young diagram $\widetilde{\lambda}$ is obtained from $\lambda$ by adding one box to the first row. For instance:
$$
\lambda ~~=~~~~
\begin{ytableau} \\
 & \\
 &  & &  \\
 &  & & & & & & \\
\end{ytableau}
\hspace{2cm}
\widetilde{\lambda} ~~=~~~~
\begin{ytableau} \\
 & \\
 &  & &  \\
 &  & & & & & & & *(white!50!blue)\\
\end{ytableau}
$$
We may assume without loss of generality that
$n=|\lambda|+1 = |\widetilde{\lambda}|$.
Then $\Hilb[n]{\X}_{\widetilde{\mu}}$ consists of only one point given by $\Z_p^{\widetilde{\mu}}$, whereas $\Hilb[n]{\X}_\mu$ is two-dimensional and its generic point is of the form $\Z_p^\mu \sqcup \{(x_1,y_1)\}$, $y_1\not=0$. The idea is to send the point $(x_1,y_1)$ to the origin along a curve that is tangent to the divisor $\{y=0\}$ to a high enough order, so that the limiting ideal will be $\Z_p^{\widetilde{\mu}}$. Here is the calculation that makes this heuristic precise.

Let us write
$$
\mathcal{I}_{\Z_p^\mu} = \Big(
x^{a_j} y^j, ~~ j = 0, 1, ..., \ell
\Big),
$$
where $a_j = \lambda_{j+1}$ for $j = 0, 1, ..., \ell$ and  $\ell = \lambda_1^T$. Let $x_1=\varepsilon$, and $y_1 = \varepsilon^{N}$, where $\varepsilon\in \mathbb{C}^*$, and $N$ is a large positive integer to be determined in a moment. Then
\begin{align*}
\mathcal{I}_{\Z_p^\mu \sqcup \{(x_1,y_1)\}} &= \Big(
x^{a_j} y^j , ~~ j = 0, 1, ..., \ell
\Big) \cap \Big(x-\varepsilon, y- \varepsilon^{N}\Big) \\
&= \Big(
x^{a_0 +1} - \varepsilon x^{a_0}, x^{a_j} y^j - \varepsilon^{b_j} x^{a_0}, j=1,2,..., \ell 
\Big),
\end{align*}
where $b_j = j N + a_j - a_0$ for $j=1,2,..., \ell$. Now choose $N$ so that $b_j > 0$ for all $j$.  Then $\Z_p^\mu \sqcup \{(x_1,y_1)\} \to \Z_p^{\widetilde{\mu}}$, as $\varepsilon\to 0$.

\emph{Case 2:} $\lambda_1 = \widetilde{\lambda}_1$.  In this case we cannot assume that $|\widetilde{\lambda}|=|\lambda|+1$, because we must maintain the condition that the Young diagrams involved are horizontally convex. For instance, in the following example there are no intermediate horizontally convex diagrams between $\lambda$ and $\widetilde{\lambda}$, even though $|\widetilde{\lambda}|=|\lambda|+3$:
$$
\lambda ~~=~~~~
\begin{ytableau}~ \\
& &\\
 & & & & \\
 &  & &  && & &  \\
\end{ytableau}
\hspace{2cm}
\widetilde{\lambda} ~~=~~~~
\begin{ytableau}~ & *(white!50!blue) \\
& &&*(white!50!blue)\\
 & & & & &*(white!50!blue)\\
 &  & &  && & & \\
\end{ytableau}
$$
We may assume without loss of generality that $n=\frac{1}{2}k(k+1)$, where $k$ is the common value of $\lambda_1$ and $\widetilde{\lambda}_1$. It is then enough to show that 
$\Hilb[n]{\X}_{\widetilde{\mu}} \cap  \mathcal{U}_{\SmallTriangularYoungDiagram}\subset \overline{\Hilb[n]{\X}_\mu} \cap  \mathcal{U}_{\SmallTriangularYoungDiagram}$, where $\mathcal{U}_{\SmallTriangularYoungDiagram}$ is the triangular chart in $(\mathbb{C}^2)^{[n]}$. But under the isomorphism $\mathcal{A}_k\cong \mathfrak{gl}_k(\mathbb{C})^\vee$, described in \autoref{cor:nilcone}, part \ref{cor:nilcone:Ak}, the symplectic leaves 
$\Hilb[n]{\X}_\mu\cap \mathcal{U}_{\SmallTriangularYoungDiagram}$ and $\Hilb[n]{\X}_{\widetilde{\mu}}\cap \mathcal{U}_{\SmallTriangularYoungDiagram}$ correspond to the conjugacy classes of nilpotent Jordan type $\mu$ and $\widetilde{\mu}$, respectively.  Recall that the condition $\lambda\subseteq \widetilde{\lambda}$ implies that $\mu\le\widetilde{\mu}$ in the dominance order.  Hence the result follows from the classical Gerstenhaber--Hesselink theorem (e.g. see \cite{O'Halloran1987}).

\emph{General case:} Note that for any pair $\lambda\subseteq \widetilde{\lambda}$, we can find an intermediate diagram $\widehat{\lambda}$ such that $\lambda\subseteq \widehat{\lambda}$ falls into Case 1 and $\widehat{\lambda}\subseteq \widetilde{\lambda}$ falls into Case 2.  This completes the proof.
\end{proof}

\subsection{Nodal points and toric degenerations}\label{sec:toric-degen}

Now suppose that $p \in \D$ is a nodal singularity.  Then by \cite{Arnold1987}, there exist coordinates $(x,y)$ such that, after rescaling $\ps$ by a nonzero constant, we have
\[
\ps =  xy \cvf{x}\wedge \cvf{y}
\] Applying the algorithm from \autoref{sec:triangular-chart-Poisson}  to compute the bracket $\hilbps = J_xJ_y\hilbps_0$, we find after a straightforward but tedious calculation that
\begin{align}\label{eq:Hilb_xy_triang}
\begin{split}
\Big\{ E_i^j, E_a^b \Big\} 
= \delta_{a\ge j} \sum_{p=0}^a E_i^{p+j-a} E_p^b - \sum_{p=i}^a E_{a+i-p}^j E_p^b -\delta_{a+1\le j} \sum_{p=a+1}^{k-1}  E_i^{p+j-a}  E_p^b + \sum_{p=a+1}^{i-1} E_{a+i-p}^j  E_p^b +
\\
+\sum_{q=0}^{\min(j,b-1 )}  E_i^{b+j-q}   E_a^q - \delta_{b\le i} \sum_{q=0}^{b-1} E_{q+i-b}^j E_a^q - \sum_{q=\max(j+1,b)}^{k}  E_i^{b+j-q}  E_a^q + \delta_{b-1\ge i} \sum_{q=b}^{k}  E_{q+i-b}^j  E_a^q .
\end{split}
\end{align}
A remarkable feature of this quadratic Poisson structure is that it admits a canonical toric degeneration, as we now explain.

\subsubsection{The torus-invariant part}
Consider the action of the torus $(\CCx)^{2n} = (\CCx)^{k(k+1)}$ on the germ of $\Utri$ by independent dilation of the \ES coordinates.    The induced action on the space of germs of bivectors preserves the space of quadratic bivectors, and hence we may project $\hilbps$ onto its torus-invariant part, which is also quadratic.
\begin{definition}
We denote by $\hilbpstor$ the torus-invariant part of $\hilbps$.
\end{definition}
Explicitly, we have 
\[
\hilbpstor = \sum_{i,j,a,b} \Pi_{ia}^{jb} ~E_i^j  E_a^b \cvf{E_i^j}\wedge\cvf{E_a^b}
\]
where the coefficients are given by
\begin{equation}\label{eq:toric_part_Hilb_xy}
  \Pi_{ia}^{jb} = \delta_{a\ge j} - \delta_{a\ge i}  - \delta_{bj}~\delta_{a\ge i+1} 
+ \delta_{b\ge j+1} ~\delta_{ai} - \delta_{b\ge j+1} + \delta_{b\ge i+1}.  
\end{equation}
This bivector is generically nondegenerate, and degenerates along the union of the coordinate hyperplanes.  The inverse log symplectic form is then given by
\[
\Hilb{\omega}_\Delta = \sum_{i,j,a,b}B_{ia}^{jb} \dlog{ E_i^j} \wedge \dlog{E_a^b}
\]
where
\begin{equation}\label{eq:Hilb_xy_bires_matrix}
B_{ia}^{jb} = -\delta_{a+b,i+j}~ \delta_{a\ge i+1} -  \delta_{a+b,i+j+1}~ \delta_{a\le i}  + \delta_{a+b,i+j} ~\delta_{a\le i-1}\ + \delta_{a+b,i+j-1} ~ \delta_{a\ge i}.
\end{equation}
is the ``biresidue'' of $\Hilb{\omega_\Delta}$ along the intersection of hyperplanes $E_{i}^j = E_{a}^b = 0$ in $\CC^{k(k+1)}$.  To understand more clearly the structure of the biresidues, it is helpful to  order the coordinates according to their position in the syzygy matrix as follows: first order by the sum of the indices, and then order by the column index to break ties, i.e. 
\begin{equation}\label{eq:ordering_Hilb_xy}
\begin{split}
    (0,0) \prec (1,0) \prec (0,1) \prec (2,0) \prec (1,1) \prec (0,2) \prec \cdots 
\end{split}
\end{equation}
For instance, the case $k=3$ is illustrated in \autoref{fig:smoothable_cycles_Hilb_xy_triang}.  
With this ordering the biresidues $B_{ia}^{jb}$ form  a skew-symmetric matrix $B$ of size $m:=k(k+1)$ with the following property, which will be useful in our study of holonomicity in \autoref{sec:holonomicity}:
\begin{definition}\label{def:cyclically_monotone}
 A skew-symmetric $m\times m$ matrix $(b_{\alpha,\beta})_{\alpha,\beta=1}^m$ is \defn{cyclically monotone} if, perhaps after multiplying all of its entries by the same nonzero constant, we have
 $$
 b_{\alpha,\alpha+1} \ge b_{\alpha,\alpha+2} \ge ... \ge b_{\alpha,m} \ge b_{\alpha,1} \ge b_{\alpha,2} \ge ... \ge b_{\alpha,\alpha-1}
 $$
 for each $1\leq \alpha\leq n$.
\end{definition}
For instance, for $k=2$, we have
$$
 B = \begin{pmatrix}
0 & 1 & 0 & 0 & 0 & 0 \\
-1 & 0 & 1 & 1 & 0 & 0 \\
0 & -1 & 0 & 1 & 1 & 0  \\
0 & -1 & -1 & 0 & 1 & 0 \\
0 & 0 & -1 & -1 & 0 & 1  \\
 0 & 0 & 0 & 0 & -1 & 0  \\
\end{pmatrix}
$$

\begin{figure}[t]
    \centering
    \includegraphics[scale=1]{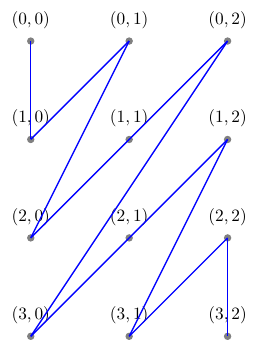}\\

\caption{Ordering the indices of a matrix}
    \label{fig:smoothable_cycles_Hilb_xy_triang}
\end{figure}

\subsubsection{Toric degeneration: a game of dominoes}

We will prove the following:
\begin{theorem}\label{thm:toric-degen}
There exists a rank-one subtorus $\G \cong \CCx \hookrightarrow (\CCx)^{(k+1)\times k}$ such that 
\[
\lim_{g \in \G, g\to 0} g \cdot \hilbps  = \hilbpstor
\]
\end{theorem}

\begin{remark}
In light of this result, we can view $\hilbps$ as a one-parameter deformation of the toric log symplectic structure $\hilbpstor$.  As explained in \cite{Matviichuk2020}, such deformations are obtained by smoothing out the nodal singularities along pairwise intersections of hyperplanes.  The combinatorics of this process can be encoded in a ``smoothing diagram'' where we draw a vertex corresponding to each hyperplane, a coloured edge joining two vertices when the corresponding intersection is smoothed, and decorating triangle with angles that indicate the order to which the smoothing degenerates along triple intersections.  In the case at hand, they hyperplanes are given by the equations $E_i^j = 0$ and are in bijection with the positions in a $(k+1)\times k$ matrix; the corresponding smoothing diagram is shown in \autoref{fig:smoothing_diagram_Hilb_xy}.\end{remark}

\begin{figure}
    \centering
    \scalebox{-1}[1]{\includegraphics[scale=0.8, angle = 90]{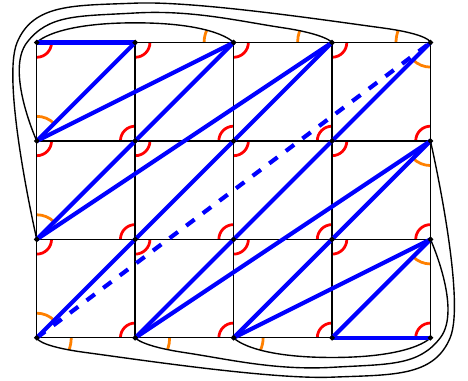}}
  \caption{Smoothing diagram encoding the deformation from $\hilbpstor$ to $\hilbps$ for a nodal curve in the triangular chart (with $k=4$).  The $k(k+1)$ hyperplanes on which $\hilbpstor$ degenerates merge into two irreducible components corresponding to the two components of the node $xy=0$; accordingly, the corresponding vertices are joined by two collections of solid blue lines.    The solid dashed line corresponds to an additional possible deformation, which is induced by smoothing $xy=0$ to $xy=\epsilon \neq 0$ in $\X$, so that the divisor becomes irreducible.}
    \label{fig:smoothing_diagram_Hilb_xy}
\end{figure}

We will prove \autoref{thm:toric-degen} by analyzing the decomposition of the space of quadratic bivectors into weight spaces for the torus action.  Namely, the theorem is equivalent to the statement that $\hilbps - \hilbpstor$ is a sum of weight vectors of the torus action, whose weights with respect to the subtorus $\G$ are strictly positive.  Note since the torus action on $\Utri$ is defined by rescaling matrix entries, it is natural to depict the weights of the action as matrices of the same size with integer coefficients.  With this convention, the weight of a monomial bivector of the form
\begin{align}
E_{i}^j E_{k}^l \cvf{E_a^b}\wedge\cvf{E_c^d} \label{eq:bivect-monomials}
\end{align}
is obtained by starting with the zero matrix, adding $1$ to the positions $(i,j)$ and $(k,l)$ in turn, and then subtracting $1$ from the positions $(a,b)$ and $(c,d)$ in turn.

In the Poisson bracket \eqref{eq:Hilb_xy_triang}, only certain weights can appear.  We will describe them in terms of the following objects.
\begin{definition}
A $(k+1)\times k$ matrix is a \defn{domino} if it has exactly two nonzero entries, one of which is $+1$ and the other of which is $-1$, and they lie either on the same row or on the same column.  We refer to the position of the entries $+1$ and $-1$ as the \defn{head} and \defn{tail} of the domino, respectively.  We say that a domino is oriented to the \defn{north} (respectively \defn{south}, \defn{east} or \defn{west}) if its head is above (resp.~below, right of, or left of) its tail.

This \defn{size} of a domino is the distance between its head and tail, and the \defn{valuation} is defined as follows:
\begin{itemize}
\item If the domino is oriented north or south, its valuation is the number of vertical translations required to move it so that its uppermost nonzero entry lies in the top column.
\item If the domino is oriented east or west, its valuation is the number of horizontal translations required to move it so that its leftmost nonzero entry lies in the leftmost column, plus $\tfrac{1}{2}$.
\end{itemize}
\end{definition}

\begin{figure}[h]
    \centering
    \includegraphics{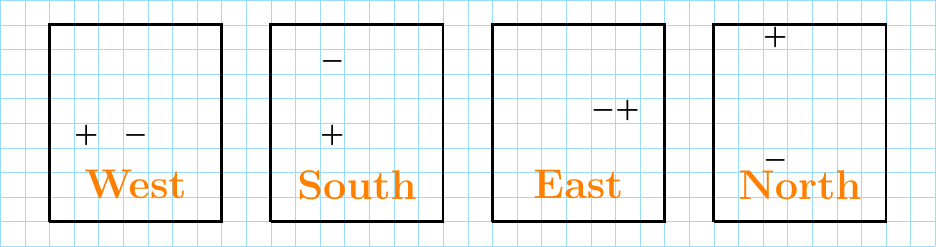}\\
    \caption{ Four examples of dominoes, one from each possible orientation.  Reading from left to right, the lengths are $2,3,1$ and $5$, and the valuations are $1+\tfrac{1}{2}$, $1$, $4+\tfrac{1}{2}$ and $0$, respectively.}
    \label{fig:domino_directions_Hilb_xy}
\end{figure}

\begin{definition}
A \defn{rectangular weight} is a $(k+1)\times k$ matrix given by a sum of two dominoes with opposite orientations, whose heads and tails form the vertices of a rectangle, with a head in the top left corner as in \autoref{fig:rect_wts_Hilb_xy}.
\end{definition}

\begin{figure}[h]
    \centering
    \includegraphics{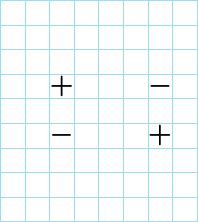}
    \caption{A rectangular weight}
    \label{fig:rect_wts_Hilb_xy}
\end{figure}

\begin{definition}
An ordered pair of dominoes is \defn{admissible} if both dominoes have the same length, the first domino is directed west or south, the second is directed east or north, and the valuation of the first is greater than the valuation of the second.
\end{definition}

By direct inspection of the formula \eqref{eq:Hilb_xy_triang}, we have the following:
\begin{lemma}
The weight of every monomial appearing in $\hilbps$ is either rectangular, or the sum of an admissible pair of dominoes.
\end{lemma}

We will reduce such weights to sums of the following elementary ones, which correspond to the weights of the first-order deformations of $\hilbpstor$ that smooth a given pairwise intersection of divisor components (i.e.~to the smoothable strata in the sense of \cite{Matviichuk2020}).

\begin{definition}
A \defn{smoothable weight} is a $(k+1)\times k$ matrix given by a sum of two dominoes of length one placed in one of the following configurations:
\begin{itemize}
\item Type I: the two dominoes are adjacent, forming a square-shaped rectangular weight
\item Type IIa: a south-east admissible pair of dominoes concentrated in the leftmost column and top row such that the tail of the first lies on the same diagonal as the head of the second.
\item  Type IIb: a west-north admissible pair of dominoes concentrated in the bottom row and rightmost column such that the head of the first lies on the same diagonal as the tail of the second.
\end{itemize}
\end{definition}
Such matrices are depicted in \autoref{fig:smooth_wts_Hilb_xy}. We remark that for each smoothable weight of type IIa or IIb, the valuation of the southwestern domino is only $\tfrac{1}{2}$ higher than the valuation of the northeastern domino, which is the minimum possible for the pair to be admissible.

\begin{figure}[h]
    \centering
    \includegraphics{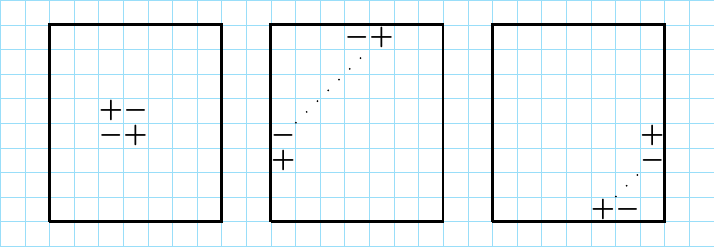}
    \caption{Smoothable weights; from left to right, the types are I, IIa and IIb}
    \label{fig:smooth_wts_Hilb_xy}
\end{figure}

\begin{lemma}\label{lem:wts-indep}
The set of smoothable weights is linearly independent.
\end{lemma}
\begin{proof}
Consider the square matrix of size $k(k+1)$ formed from the biresidues $B_{ip}^{jq}$ \eqref{eq:Hilb_xy_bires_matrix} of the toric log symplectic form with indices ordered as in \eqref{eq:ordering_Hilb_xy}.  Then every smoothable weight can be expressed as the difference of a unique pair of consecutive rows of $B_{ip}^{jq}$.  But this matrix is invertible, and hence the differences of its consecutive rows are linearly independent.
\end{proof}

\begin{lemma}\label{lem:rect-smooth}
Any rectangular weight is a sum of smoothable weights of Type 1 with non-negative integer coefficients.
\end{lemma}

\begin{proof}
Consider the sum of all smoothable weights of Type I that are contained inside the rectangle; this sum telescopes so that only the vertices of the rectangle remain.
\end{proof}

\begin{proposition}\label{prop:dom-smooth}
The sum of any admissible pair of dominoes is a linear combination of smoothable quadratic weights with non-negative integer coefficients.
\end{proposition}

\begin{proof}
The proof is a sort of game, in which we translate dominoes in the plane by certain admissible operations, which correspond mathematically to subtracting collections of smoothable weights.  We will show that by repeated application of such moves, every admissible pair may be reduced to zero.

Firstly, by subtracting a rectangular weight, we can push any southern domino all the way to the leftmost column, as illustrated in \autoref{fig:pushing_east_domino_up}. 
\begin{figure}[t]
    \centering
    \includegraphics{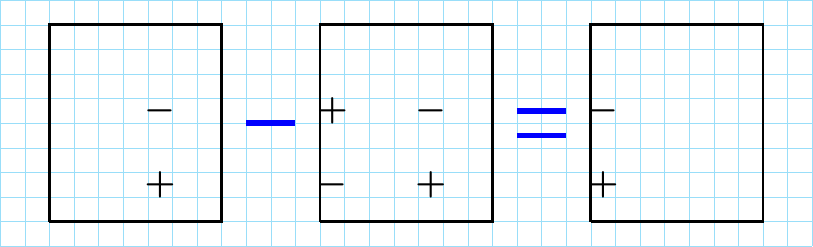}
    \caption{Pushing a southern domino to the leftmost column}
    \label{fig:pushing_east_domino_up}
\end{figure}
Therefore, whenever we have an admissible pair of dominoes, one of which is southern, we can assume without loss of generality that the southern domino is located in the leftmost column. Likewise, we can assume that every northern domino in an admissible pair is located in the rightmost column, every eastern domino in the top row, and every western domino in the bottom row (just as the dominoes in the smoothable weights of type II).

Secondly, we claim that by subtracting several rectangles and smoothable weights of types IIa and IIb, we can move any southern domino in the leftmost column one step to the north.  For this, we proceed in three steps.  In the first step, we subtract some smoothable weights of type IIa to turn the domino to the west, as in \autoref{fig:pushing_east_domino_left_step1}.  In the second step, we subtract a rectangular weight to push the obtained western domino to the bottom row as in \autoref{fig:pushing_east_domino_left_step2}.  Finally, in the third step, we subtract some smoothable weights of type IIb to turn the western domino back to the south, as in \autoref{fig:pushing_east_domino_left_step3}.  Because the matrix has one more row than it has columns, the domino is now one step northwards from where it started, as desired.

\begin{figure}[t]
    \centering
    \includegraphics{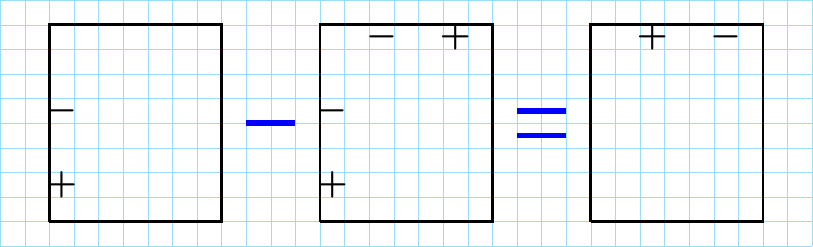}
    \caption{Pushing a southern domino to the north, step 1.}
    \label{fig:pushing_east_domino_left_step1}
\end{figure}

\begin{figure}[t]
    \centering
    \includegraphics{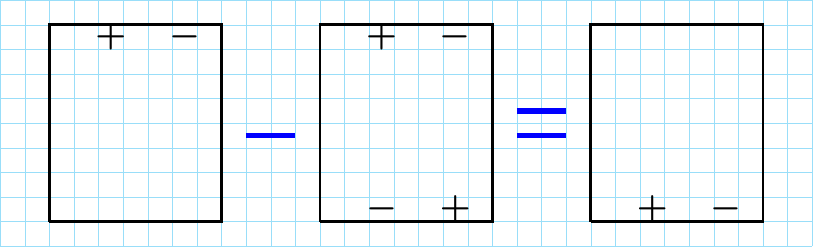}\\
    \caption{Pushing a southern domino to the north, step 2.}
    \label{fig:pushing_east_domino_left_step2}
\end{figure}
 
\begin{figure}[t]
    \centering
    \includegraphics{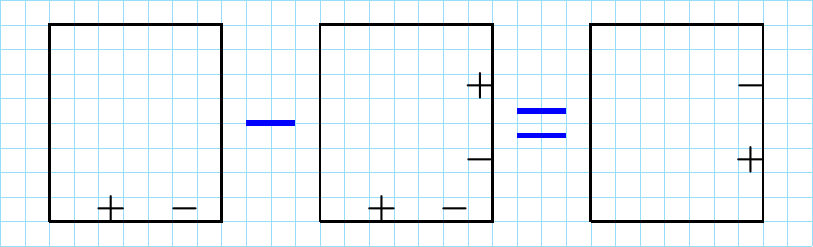}
    \caption{Pushing an eastern domino to the west, step 3.}
    \label{fig:pushing_east_domino_left_step3}
\end{figure}
Using these operations, we are free to push any southern domino in an admissible pair to the north, without changing the second domino in the pair, as long as the pair stays admissible. Similarly, we may push any western domino to the east, any northern domino to the south and any eastern domino to the west.

Now by definition, there are four types of admissible pairs of dominoes: south-east, south-north, west-north, and west-east.  Note that by subtracting smoothable weights of type II as in \autoref{fig:pushing_east_domino_left_step1}, we can convert any admissible south-north pair into an admissible west-north pair, and similarly we may convert any  admissible west-east pair into an admissible south-east pair, so it suffices to treat the south-east and west-north cases.

To this end, consider a south-east pair and assume without loss of generality that the southern domino is in the left column, while the eastern one is in the top row. Pushing the southern one as far as possible to the north, we can reduce its valuation to the minimum value for which the pair remains admissible.  In this way we obtain an admissible pair that is evidently a telescoping sum of smoothable weights of Type IIa.  Similarly, pushing a west-north pair to the bottom row and right column, and moving the northern domino south to maximize its valuation, we obtain a sum of smoothable weights of type IIb.
\end{proof}

We can now complete the construction of the toric degeneration:
\begin{proof}[Proof of \autoref{thm:toric-degen}]
By \autoref{lem:wts-indep}  there exists  a coweight $w$ of the torus $(\mathbb{C}^*)^{(k+1)k}$  that pairs with each smoothable weight  to give a positive integer.  By construction, the corresponding one-parameter subgroup $\G\cong \CCx$ then acts with positive weight on every smoothable weight space, and hence by \autoref{lem:rect-smooth} and \autoref{prop:dom-smooth} acts with positive weight on the weights spaces corresponding to rectangular weights or admissible pairs of dominoes.
\end{proof}

\subsection{Holonomicity}
\label{sec:holonomicity}

Recall from \cite{Pym2018} that every holomorphic Poisson manifold $(\W,\ps)$ has an associated \defn{characteristic variety}
\[
\charps \subset \ctb{\W}
\]
defined as the singular support of the complex of $\sD{\W}$-modules associated to the deformation complex of $(\W,\ps)$.  We say that $(\W,\ps)$ is \defn{holonomic} if its characteristic variety is Lagrangian.  This is a local property that depends only on the stable equivalence classes of the germs of $\ps$.
It was shown in \emph{op.~cit.} that if $(\W,\ps)$ is holonomic then every point has a neighbourhood with only finitely many characteristic symplectic leaves and that the converse holds when $\dim \W = 2$.  In \cite{Matviichuk2020}, we conjectured that the converse holds in all dimensions, and we proved this when $(\W,\ps)$ is generically symplectic and degenerates along a normal crossing divisor.  In this section, we give further evidence for the conjecture.

Considering the invariance under stable equivalence and the results in \autoref{sec:char-leaves}, there are three cases one must consider in order to verify the conjecture for all Hilbert schemes: either (a)
 $\D$ has at worst nodal singularities and $n$ is arbitrary; (b) $\D$ has only double points and $n=2$; or (c) $\D$ has only $A_2$ singularities and $n=5$.   We shall treat the cases (a) and (b); we do not know whether the conjecture holds in case (c) but have no reason to doubt it.  Case (a) is the subject of the following theorem, whose proof will occupy most of this section.  Note that it implies that $\Hilb{\X}$ is holonomic if $\D$ is smooth; a more direct proof is possible in this case, but we omit it.  The smooth case allows us to immediately treat case (b) as a corollary.
 
\begin{theorem}\label{thm:nodal-holonomic}
For a Poisson surface $\X$ whose vanishing locus is the anticanonical divisor $\D\subset \X$, the following statements are equivalent:
\begin{enumerate}
\item For every $n\ge0$, the induced Poisson structure on the Hilbert scheme $\Hilb[n]{\X}$ is holonomic.
\item  For every $n\ge0$, the germ of  $\Hilb[n]{\X}$ at any point has only finitely many characteristic symplectic leaves.
\item The only singularities of $\D$ are nodes.
\end{enumerate}
\end{theorem} 
\begin{corollary}\label{cor:hilb2-holonomic}
Let $\X$ be a Poisson surface with reduced anticanonical divisor $\D\subset \X$. Then $\Hilb[2]{\X}$ is holonomic if and only if every singular point of $\D$ is a double point.
\end{corollary}

\begin{proof}[Proof of \autoref{cor:hilb2-holonomic}]
Let $\V \subset \Hilb[2]{\X}$ be the set of elements set-theoretically supported at a single singular point of $\D$.  Then $\V$ is the disjoint union of several copies of $\PP^1$, one for each singular point, and $\Hilb[2]{\X}\setminus \V$ is holonomic by \autoref{thm:nodal-holonomic}.  It therefore suffices to show that $\charhilbps[2] \cap (\ctb{\Hilb[2]{\X}})|_{\V}$ has pure dimension four. But since all singular points are double points, \autoref{prop:char-leaves} implies that the zeros of the modular vector field in $\V$ are isolated, so it defines a function on the smooth five-dimension variety $\ctb{\Hilb[2]{\X}}|_\V$ that is nonzero on every connected component.  Its vanishing locus contains $\charhilbps[2] \cap \ctb{\Hilb[2]{\X}}|_{\V}$ by \cite[Theorem 3.4, part 1]{Pym2018}, and hence the latter has dimension four, as desired.
\end{proof}

We now proceed with the proof of \autoref{thm:nodal-holonomic}. The equivalence of statements 2 and 3 was proven in \autoref{sec:char-leaves}, so in light of the above it remains only to prove that if $\D$ is nodal, then $\Hilb{\X}$ is holonomic for all $n$.  The problem is local and invariant under stable equivalence and taking products, so by \autoref{lm:Hilb_equiv_to_triang_chart}, it suffices to prove these statements for the triangular chart.  Moreover, in light of \autoref{lm:holonomicity_openness} below and \autoref{thm:toric-degen}, the problem reduces to proving the statement for the toric degeneration $\hilbpstor$:

\begin{lemma}\label{lm:holonomicity_openness}
Let $\pi_t$, $t\in \mathbb{C}$, be a polynomial family of quadratic Poisson tensors on $\mathbb{C}^{2n}$ such that $\pi_0$ is holonomic. Then $\pi_t$ is holonomic for all but finitely many $t \neq 0$.
\end{lemma}

\begin{proof}
Let $\V_t\subset \tb{\CC^{2n}} \cong \CC^{4n}$ be the characteristic variety of the complex of $\mathcal{D}$-modules $\mathcal{M}^\bullet_{\pi_t}$, which is thus invariant under the action of $\CCx$ by dilation of the fibres.   Additionally, since all Poisson structures $\pi_t$ are quadratic, each $\V_t$ is also invariant under the dilation on the base $\mathbb{C}^{2n}$. Combining these two symmetries, we get that all $\V_t$ are invariant under the uniform dilation of all directions $\mathbb{C}^{4n}$. In other words, the projectivization $\PP(\V_t)\subset \PP^{4n-1}$ is well-defined for each $t$. Since $\pi_t$ is assumed holonomic, we have $\dim_\CC \PP(\V_t) = 2n-1$. Then by semicontinuity of dimension for proper maps, we have that $\dim_\CC \PP(\V_t)\le 2n-1$ for $t$ in a Zariski neighbourhood of the origin, as desired.
\end{proof}

The toric degeneration $\hilbpstor$ is a log symplectic manifold with normal crossings boundary, so the results of our earlier work~\cite{Matviichuk2020} apply.  In particular,  by condition 4 of \cite[Theorem 1.5]{Matviichuk2020}, and by the cyclic monotonicity property of the biresidues of $\hilbpstor$ (\autoref{def:cyclically_monotone}), it suffices to prove the following linear-algebraic statement:

\begin{lemma} \label{lm:cyclically_monotone_holonomic}
Let $B$ be a cyclically monotone matrix of odd size, with only zeros and ones above the diagonal.
Then the row span of $B$ does not contain the constant vector $(1,1,...,1)\in \mathbb{C}^m$.
\end{lemma}
Although the statement is elementary, we did not manage to find an easy direct proof.  We will instead prove a more general statement, that is amenable to an inductive argument.  For a tuple $J = (J_1, \ldots, J_m)$ of non-empty open intervals $J_1=(c_1,d_1),\ldots,J_m=(c_m,d_m) \subset \RR$, such that all the endpoints of $J_i$ are distinct, let 
\[
\ell(J) := (d_1-c_1,d_2-c_2,\ldots,d_m-c_m)
\]
be the vector of their lengths.  Let us say that \defn{$J_\alpha$ overlaps   $J_\beta$ on the left} (respectively \defn{right}) if $c_\alpha < c_\beta < d_\alpha < d_\beta$ (resp.~$c_\beta < c_\alpha < d_\beta < d_\alpha$).  Define an $m\times m$ skew-symmetric matrix $B(J)$ by the formula
\begin{align}
B(J)_{\alpha\beta} = \begin{cases}
1 & \textrm{if }J_\alpha \textrm{ overlaps }J_\beta \textrm{ on the left} \\
-1 & \textrm{if }J_\alpha \textrm{ overlaps }J_\beta \textrm{ on the right} \\
0 & \textrm{otherwise}.
\end{cases}
\end{align}

\begin{lemma}
Every matrix $B$ as in \autoref{lm:cyclically_monotone_holonomic} is of the form $B = B(J)$ for some collection $J$ of intervals of lengths $\ell(J) = (1,\ldots,1)$.
\end{lemma}

\begin{proof}
For $1 \le \beta \le m$, let $k_\beta$ be the row index  of the uppermost $1$ in the $\beta$th column of $B$, with the convention that if there are no ones in the $\beta$th column, we define $k_\beta=\beta$. Set $J_1 = (0,1)$, and choose intervals  $J_2,\ldots, J_m$ of length one inductively by choosing the left endpoint  $c_{\beta} > c_{\beta-1} + 1$ if $k_\beta = \beta$ and $c_\beta \in (\max\{c_{\beta-1},c_{k_\beta -1}+ 1\}, c_{k_\beta} + 1)$ if $k_\beta < \beta$.  Then $\ell(J) = (1,\ldots,1)$ and $B(J) = B$, as desired.
\end{proof}

\begin{proposition}
Let $J =(J_1,\ldots,J_m)$ is a tuple of intervals, where $m$ is odd and the endpoints of $J_i$ are distinct.  Then the vector $\ell(J)  \in \RR^m$ of their lengths does not lie in the row span of $B(J)$.
\end{proposition}
\begin{proof}
We  proceed by induction on the odd integer $m$.  The base case $m=1$, is obvious, so assume the proposition holds for some odd integer $m-2$; we will prove that it holds for  $m$.  

Let view $B(J)$ as a bivector $B(J) \in \wedge^2 \mathbb{R}^m$ so that we have
$$
B(J) := \sum_{1\le \alpha<\beta\le m} B(J)_{\alpha,\beta}~ e_\alpha \wedge e_\beta \qquad \ell(J) = \sum_{\alpha=1}^m \ell(J)_\alpha\,e_\alpha
$$
where $e_1,e_2,\ldots$ is the standard basis of $\mathbb{R}^m$.  

By permuting the indices, we may assume without loss of generality that the sequence of left endpoints of the intervals is strictly increasing. Then if $\alpha < \beta$, the interval $J_\alpha$ can never overlap $J_\beta$ on the right, and hence $B(J)_{\alpha,\beta} \in \{0,1\}$ for $\alpha<\beta$.  

If no interval $J_\gamma$ with $\gamma < m$ overlaps $J_m$, then $B(J)_{\alpha,m}=0$ for all $\alpha$, and since $\ell(J)_m > 0$ we conclude that $\ell(J)$ is not in the row span of $B(J)$, as desired.  Hence we may assume without loss of generality that there exists a maximal index $\gamma < m$ such that $J_\gamma$ overlaps $J_m$ on the left. Define a new basis $ \te_1,\ldots, \te_m$ for $\RR^m$ by the formula
\[
\te_\alpha := \begin{cases}
e_{\gamma} - \sum_{\beta\not=\gamma,m} B(J)_{m,\beta}\cdot e_\beta & \alpha = \gamma \\
e_m + \sum_{\beta\not=\gamma,m} B(J)_{\gamma,\beta}\cdot e_\beta  & \alpha = m \\
e_\alpha & \textrm{otherwise}.
\end{cases}
\]
It is tedious but straightforward to check that in this new basis, we have
\[
B(J) = \te_\gamma \wedge \te_m  + \sum_{\substack{\alpha,\beta,\gamma, m \textrm{ distinct} \\ \alpha<\beta} } B(\tJ)_{\alpha,\beta} \te_\alpha \wedge \te_\beta  \qquad \ell = a_1 \te_\gamma + a_2 \te_m + \sum_{\alpha \neq \gamma,m} \ell(\tJ)_\alpha \te_\alpha
\]   
where $a_1,a_2 \in \ZZ$, and $\tJ = (T \cdot J_1,\ldots,\widehat{T \cdot J_\gamma},\ldots, T \cdot J_{m-1})$ is the collection of $m-2$ intervals defined as follows.  First, define the ``interval exchange transformation'' to be the bijection $T : \RR \to \RR$ which translates the intervals between the endpoints of $J_\gamma=(c_\gamma,d_\gamma)$ and $J_m=(c_m,d_m)$ as shown in \autoref{fig:interval_exch}, and acts as the identity on the rest of the real line.  Then $T \cdot J_\alpha = (T(c_\alpha),T(d_\alpha))$ is the interval spanned by the $T$-images of the endpoints of $J_\alpha$. By induction, $\ell(\tJ)$ is not in the image of $B(\tJ)$, and hence $\ell(J)$ is not in the image of $B(J)$, as desired.
\end{proof}

\begin{figure}[t]
    \centering
    \includegraphics{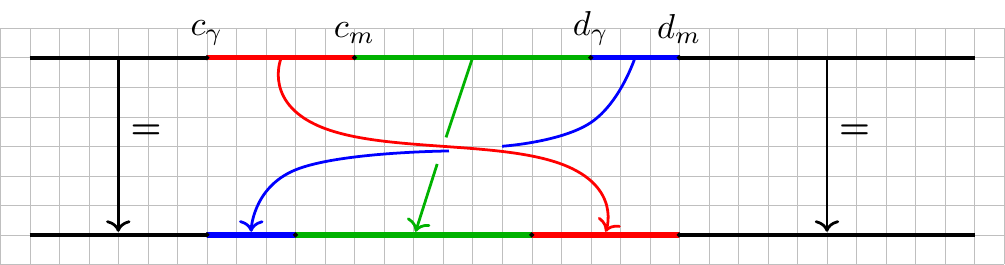}
    \caption{The interval exchange transformation $T$}
    \label{fig:interval_exch}
\end{figure}

\section{Deformation theory}
\label{sec:def-thy}

In this section we address some aspects of the deformation theory of Bottacin's Poisson structures.  To this end, recall that if $(\W,\ps)$ is a holomorphic Poisson manifold, then the sheaf $\der{\W}=\wedge^\bullet{\cT{\W}}$ of polyvector fields comes equipped with a differential $\dps := [\ps,-]$.  The hypercohomology $\Hps{\W}$ of this complex (called the Poisson cohomology) controls deformations of $\W$ as a holomorphic Poisson manifold, or more generally, as a generalized complex manifold.

 We shall compute the Poisson cohomology of $\Hilb{\X}$ in low degrees, and in particular determine the deformation space, under the assumption that $\D$ is reduced, and has only quasi-homogeneous singularities (i.e.~is locally the zero set of a quasi-homogeneous polynomial).  As remarked in \cite[Section 4.4]{Pym2018}, the quasi-homogeneity is automatic if $\X$ is projective, thanks to the classification in \cite[Section 7]{Ingalls1998}, or by analysis of the minimal models and their blowups.

\subsection{Surface case}
Let us briefly recall the description of the Poisson cohomology in the case $n=1$ from \cite[Section 4.4]{Pym2018} and \cite{Goto2015}.
Let us denote by $\U := \X\setminus \D$ the open symplectic leaf, and by
\[
j : \U \hookrightarrow \X
\]
its open embedding.  Since $\ps$ is symplectic over $\U$, we have a canonical isomorphism
\[
j^*(\der{\X},\dps) \cong (\forms{\U},\dd) \cong \CC_\U
\]
in the constructible derived category of $\U$, which gives, by adjunction, a natural map
\begin{align}
(\der{\X},\dps) \to Rj_*\CC_\U \label{eq:surface-open-restrict}
\end{align}
in the constructible derived category of $\X$.

On the other hand, let $\D_\sing \subset \D$ denote the scheme-theoretic singular locus of $\D$, cut out locally by the vanishing of a defining equation for $\D$ and its partial derivatives, and let
\[
\xymatrix{
i : \D_\sing \hookrightarrow \X
}
\]
be the corresponding closed embedding.  Its image is the union of all codimension-two characteristic symplectic leaves in $\X$.  One easily checks that the restriction of bivectors on $\X$ to $\D_\sing$ annihilates the image of $\dps$, and hence there is a canonical map
\begin{align}
(\der{\X},\dps) \to i_*i^*\acan{\X}[-2], \label{eq:surface-sing-restrict}
\end{align}
where the right hand side is a skyscraper sheaf supported on $\D_\sing$, viewed as a complex concentrated in degree two with trivial differential.  Its stalk at a point $p \in \D_\sing$ is canonically identified with the tangent cohomology $\coH[1]{\tc[\D,p]}$, i.e.~the space of smoothings of the singularities of $\D$ at $p$. 

The maps \eqref{eq:surface-open-restrict} and \eqref{eq:surface-sing-restrict} determine the Poisson cohomology as follows:
\begin{theorem}[{\cite[Theorem 4.7]{Pym2018}}]\label{thm:surface-Hps}
If $\D$ is reduced with only quasi-homogeneous singularities, then the canonical map
\[
(\der{\X},\dps) \to Rj_*\CC_\U \oplus i_*i^*\acan{\X}[-2]
\]
is a quasi-isomorphism.  In particular,
\[
\Hps{\X} \cong \coH{\U;\CC} \oplus \coH[0]{\D_\sing, i^*\acan{\X}}[-2].
\]
as graded vector spaces.
\end{theorem}

\subsection{Higher-dimensional case}

We now use the results in the previous section to treat the case $n > 1$, as follows. We have the natural open embedding
\[
\Hilb{j} : \Hilb{\U} \hookrightarrow \Hilb{\X},
\]
identifying $\Hilb{\U}$ with the open symplectic leaf in $\Hilb{\X}$.  This gives a natural map
\begin{align}
(\der{\Hilb{\X}},\dhilbps) \to Rj_*\CC_{\Hilb{\U}} \label{eq:hilb-open-restrict}
\end{align}
Meanwhile we have a Poisson rational map
\[
\Hilb[n-1]{\X}\times \X \dashrightarrow \Hilb{\X}
\]
given by $(\Z,p) \mapsto \Z \sqcup \{p\}$ whenever $p \notin \Z$.  Since $\U$ and $\D_\sing$ are disjoint by definition, this map restricts to an embedding
\[
\Hilb{i} : \Hilb[n-1]{\U}\times \D_{\sing} \hookrightarrow \Hilb{\X}
\]
with trivial normal bundle.  By \autoref{cor:codim2_char_leaves}, the image of this embedding is the union of the codimension-two characteristic symplectic leaves of $\Hilb{\X}$.  Since $\Hilb[n-1]{\U}$ is symplectic, we have by K\"unneth decomposition and \autoref{thm:surface-Hps} that
\begin{align*}
(\Hilb{i})^*(\der{\Hilb{\X}},\dhilbps) &\cong \CC_{\Hilb[n-1]{\U}} \boxtimes i^*(Rj_*\CC_\U \oplus \acan{\X}[-2]) \\
&\cong (\Hilb{i})^* R\Hilb{j}_* \CC_{\Hilb{\U}} \oplus (\CC_{\Hilb[n-1]{\U}} \boxtimes i^*\acan{\X}[-2])
\end{align*}
where $\boxtimes$ denotes the external tensor product of sheaves on the product $\Hilb[n-1]{\U}\times \D_\sing$.  Projecting onto the second summand, we obtain a natural map
\begin{align}
(\der{\Hilb{\X}},\dhilbps) \to \Hilb{i}_* \rbrac{\CC_{\Hilb[n-1]{\U}} \boxtimes i^*\acan{\X}}[-2] \label{eq:hilb-sing-restrict}
\end{align}

Combining \eqref{eq:hilb-open-restrict} and \eqref{eq:hilb-sing-restrict}, we obtain a canonical map
\begin{align}
\phi : (\der{\Hilb{\X}},\dhilbps) \to  Rj_*\CC_{\Hilb{\U}} \oplus \Hilb{i}_* \rbrac{\CC_{\Hilb[n-1]{\U}} \boxtimes i^*\acan{\X}}[-2] \label{eq:Hps-hilb-decomp}
\end{align}
\begin{theorem}\label{thm:hilb-Hps}
The induced map
\[
\coH[j]{\phi} : \Hhilbps[j]{\X} \to \coH[j]{\Hilb{\U};\CC} \oplus \rbrac{\coH[j-2]{\Hilb[n-1]{\U};\CC}\otimes \coH[0]{i^*\acan{\X}}}
\]
is an isomorphism for $j \le 2$.
\end{theorem}

Before giving the proof, which will occupy the rest of this section, let us note that the cohomology $\coH{\Hilb[n]{\U};\CC}$ can be computed from $\coH{\U;\CC}$ using the G\"ottsche--Soergel formula~\cite{Goettsche1993}.  In particular, applying the theorem to the degree two hypercohomology, we obtain the following
\begin{corollary}
If $\D$ is reduced and has only quasi-homogeneous singularities, then the space of first-order deformations of $(\Hilb{\X},\hilbps)$ is given, for all $n \ge 2$ by
\begin{align*}
\Hhilbps[2]{\X} &\cong  \coH[2]{\Hilb{\U};\CC} \oplus \coH[0]{i^*\acan{\X}} \\
&\cong\rbrac{\coH[2]{\U;\CC} \oplus \wedge^2 \coH[1]{\U;\CC} \oplus  \CC\cdot[\E]} \oplus \coH[0]{i^*\acan{\X}} \\
&\cong \Hps[2]{\X} \oplus \wedge^2 \coH[1]{\U;\CC} \oplus \CC\cdot[\E]
\end{align*}
where $[\E] \in \coH[2]{\Hilb{\U};\CC}$ is the first Chern class of the exceptional divisor of the Hilbert--Chow morphism.
\end{corollary}

Our proof of \autoref{thm:hilb-Hps} follows the strategy used in \cite[Section 4]{Pym2018} and \cite{Matviichuk2020}.  Namely, let 
\[
\Hilb{\D} := \Hilb{\X} \setminus \Hilb{\U} \cong \set{\Z \in \Hilb{\X}}{\Z \cap \D \neq \varnothing}
\]
be the degeneracy divisor of the Poisson structure on $\Hilb{\X}$.  (Note that despite the notation, it is \emph{not} isomorphic to the Hilbert scheme of $\D$.)  Let $\hilblogforms{\X}{\D}$ be the sheaf of logarithmic differential forms in the sense of K.~Saito~\cite{Saito1980}.  Then, as for any log symplectic manifold, the natural map $\ps^\sharp : \forms[1]{\X} \to \cT{\X}$ induces a commutative triangle
\begin{align}
\xymatrix{
& (\hilblogforms{\X}{\D},\dd) \ar[ld] \ar[rd] \\
 (\der{\X},\dps)  \ar[rr] & &Rj_*\CC_\U 
} \label{eq:log-triangle}
\end{align}
where all maps are quasi-isomorphisms away from the singular locus $\Hilb{\D}_\sing$  of $\Hilb{\D}$; see~\cite[p.~695]{Pym2018}.  

For our purposes, it suffices to understand these complexes along the generic part of the singular locus.  To this end, define locally closed subvarieties $\bS_1,\bS_2 \subset \Hilb{\D}_\sing$ by
\begin{align*}
\bS_1 &:= \set{ \Z \in \Hilb{\X} }{ \Z\cap \D \textrm{\ consists of two reduced points} } \\
\bS_2 &:= \Hilb{i}(\Hilb[n-1]{\U}\times \D_\sing).
\end{align*}
Evidently $\bS_1$ and $\bS_2$ are disjoint.   Moreover, we have the following.
\begin{lemma}\label{lem:codim-3-stratum}
The subset
\[
\W := \Hilb{\X} \setminus \rbrac{\Hilb{\U} \sqcup \Hilb{\D}_\reg \sqcup \bS_1\sqcup \bS_2}
\]
is a closed subvariety of codimension three. 
\end{lemma}

\begin{proof}
If $\Z \in \W$ then $\Z$ contains at least one of the following configurations: three distinct points of $\D$; a singular point and a smooth point; or a nonreduced point tangent to $\D$. Each of these conditions defines a locally closed subvariety of $\Hilb{\X}$ that maps birationally onto its image in $\sympow{\X}$, which clearly has codimension three.
\end{proof}

\begin{lemma}
The morphism $\phi$ from \eqref{eq:Hps-hilb-decomp} restricts to a quasi-isomorphism over $\Hilb{\X}\setminus \W$.
\end{lemma}

\begin{proof}
It is already a quasi-isomorphism away from $\Hilb{\D}_\sing$, and is furthermore a quasi-isomorphism over the image of $\Hilb{i}$, so  it suffices to show that it is an isomorphism over $\bS_1$.  Since $\bS_1$  is disjoint from the support of $\Hilb{i}_*(\CC_{\Hilb[n-1]{\U}} \boxtimes i^*\acan{\X})$, this is equivalent to showing that the map \eqref{eq:hilb-open-restrict} is a quasi-isomorphism there.  But in a neighbourhood of this locus, $\Hilb{\D}$ has only normal crossings singularities, and the only characteristic leaf is the open one.   Hence the result follows from our computation of the Poisson cohomology in the normal crossings case~\cite{Matviichuk2020}; see also \cite{Ran2017}.
\end{proof}

Since $\W$ has codimension three and $\der{\X}$ is locally free, the restriction map from Poisson cohomology of $\Hilb{\X}$ to that of $\Hilb{\X\setminus \W}$  is an isomorphism in degrees 0 and 1, and injective in degree 2, by a standard local cohomology argument (a generalization of Hartogs' principle).  Since $\Hilb[n]{\U}$ and $\Hilb{i}(\Hilb[n-1]{\U}\times \D_\sing)$ are disjoint from $\W$ we have the following.
\begin{corollary}
The map $\coH[j]{\phi}$ on hypercohomology is an isomorphism in degree 0 and 1, and injective in degree 2.
\end{corollary}

Hence to prove \autoref{thm:hilb-Hps}, it remains to show that the map $\coH[2]{\phi}$ is surjective. To this end, we require two lemmas:
\begin{lemma}
If $\D$ has only quasi-homogeneous singularities, then the natural map
\[
\coH[j]{\hilblogforms{\X}{\D}} \to \coH[j]{\Hilb{\U};\CC}
\]
is surjective for $j \le 2$.
\end{lemma}

\begin{proof}
The argument is a de Rham analogue of the approach of G\"ottsche--Soergel~\cite{Goettsche1993}.  We have the natural maps
\[
\xymatrix{
\Hilb{\U} \ar[d]^{\Hilb{j}} \ar[r]^-{} & \sympow{\U}\ar[d]^-{\sympow{j}} & \ar[l]_{} \U^n \ar[d]^-{j^n} \\
\Hilb{\X} \ar[r]^-{r} & \sympow{\X} & \ar[l]_{q} \X^n 
}
\]
and corresponding divisors $\Hilb{\D}, \sympow{\D}$ and $\D^n$; for instance $\D^n$ is the set of $n$-tuples $(p_1,\ldots,p_n) \in \X^n$ such that $p_i \in \D$ for some $i$.

Since  $\sympow{\X} = \X^n/\bS_n$ is a finite quotient, it has only rational klt singularities, and hence by \cite[Lemma 1.8]{Steenbrink1977} and \cite[Theorem 1.4]{Greb2011} we have 
\[
r_*\forms{\Hilb{\X}} = (q_*\forms{\X^n})^{\bS_n}
\]
where on the right-hand side we have taken invariants with respect to the symmetric group action on $\X^n$.  An identical statement holds replacing $\X$ everywhere with $\U$.  Note that since the exceptional divisor $\E\subset \Hilb{\X}$ is not contained in $\Hilb{\D}$, a form $\omega \in \forms{\Hilb{\U}}$ extends to a logarithmic form on $\Hilb{\X}$ if and only if it extends to a logarithmic form on $\Hilb{\X} \setminus \E$ that has no poles on $\E$.  Since the maps $r$ and $q$ are \'etale away from $r(\E)$ and relate the corresponding divisors, we conclude that
\[
(r_*\hilblogforms{\X}{\D},\dd) \cong (q_*\logforms{\X^n}{\D^n})^{\bS_n},\dd) \cong (q_*R(j^n_*) \CC_{\U^n})^{\bS_n} \cong R\sympow{j}_*\CC_{\sympow{\U}}
\]
where the second isomorphism is a consequence of the K\"unneth decomposition and the results of~\cite{Castro-Jimenez1996} (using that $\D$ is quasi-homogeneous). 
 
Passing to hypercohomology, we obtain a commutative diagram
\[
\xymatrix{
\coH{r_*\hilblogforms{\X}{\D},\dd}  \ar[r] \ar[d]_-{\sim} & \coH{\hilblogforms{\X}{\D},\dd}  \ar[d]  & \coH{\forms{\Hilb{\X}}} \ar[l]_-{} \ar[d]^-{\sim} \\
\coH{\U^n;\CC}^{\bS_n} \cong \coH{\sympow{\U};\CC} \ar[r] & \coH{\Hilb{\U};\CC} & \ar[l]^-{(\Hilb{j})^*} \coH{\Hilb{\X};\CC}
}
\]
where the bottom left arrow arrow is the inclusion given by the codimension-zero contribution to the G\"ottsche--Soergel formula.  Said map is an isomorphism in degrees zero and one, and in degree two it is an injection with complement spanned by the class of the exceptional divisor $\E \subset \Hilb{\U}$.  Since the latter extends canonically to $\X$, it follows that the map $\coH{\hilblogforms{\X}{\D},\dd} \to \coH{\Hilb{\U};\CC}$ is surjective in degree two.
\end{proof}

Now let us observe that by applying Bottacin's construction in families, any Poisson deformation of $(\X,\ps)$ induces a Poisson deformation of $(\Hilb{\X},\ps)$.  Under this correspondence, deformations of the singularities of $\D\subset \X$ induce identical deformations of the singularities of $\Hilb{\D}\subset\Hilb{\X}$ transverse to the codimension-two characteristic leaves.  Applied to first-order deformations, this means that we have a canonical morphism $\Hps[2]{\X} \to \Hps[2]{\Hilb{\X}}$ such that the composition
\[
\xymatrix{
\coH[0]{i^*\acan{\X}} \ar@{^(->}[r] & \Hhilbps[2]{\X} \ar[r]^-{\phi} & \coH[2]{\Hilb{\U};\CC} \oplus \coH[0]{i^*\acan{\X}} \ar[r] &  \coH[0]{i^*\acan{\X}}
}
\]
is the identity map.  Furthermore, the image of $\coH[0]{i^*\acan{\X}}$ in $\Hhilbps[2]{\X}$ is complementary to the image of $\coH[2]{\hilblogforms{\X}{\D}}$, because the latter can only produce deformations in which the divisor $\Hilb{\D}$ deforms locally trivially.  It follows that $\coH[2]{\phi}$ is surjective, as desired.  This completes the proof of \autoref{thm:hilb-Hps}.

\bibliographystyle{hyperamsplain}
\bibliography{hol-hilb}

\end{document}